\documentclass[12pt,reqno]{amsart}
\usepackage{enumitem}
\usepackage{amssymb}
\usepackage{graphicx}
\usepackage{amsfonts}
\usepackage{amsmath}
\usepackage[unicode]{hyperref}
\usepackage{verbatim}
\usepackage{url}
\usepackage[usenames]{color}
\usepackage{stmaryrd}
\usepackage{latexsym}
\usepackage{prooftree}
\usepackage[all]{xy}

\newif\ifarxiv
\arxivtrue

\makeatletter
\@namedef{subjclassname@2020}{%
  \textup{2020} Mathematics Subject Classification}
\makeatother

\newtheorem{thm}{Theorem}[section]
\newtheorem{lemma}[thm]{Lemma}
\newtheorem{proposition}[thm]{Proposition}

\newtheorem{fact}[thm]{Fact}

\theoremstyle{definition}
\newtheorem{definition}[thm]{Definition}
\newtheorem{remark}[thm]{Remark}

\numberwithin{equation}{section}

\DeclareRobustCommand{\VAN}[3]{#2} 

\newcommand\Sober{\mathcal S}
\newcommand\Hoare{\mathcal H}

\newcommand\limp{\mathrel{\Rightarrow}}
\newcommand\nat{\mathbb{N}}

\newcommand\pow{\mathop{\mathbb P}}
\newcommand\Pfin{\pow_{fin}}
\newcommand\dom{\mathop{\mathrm{dom}}}

\newcommand\upc{\mathop{\uparrow}\nolimits}
\newcommand\dc{\mathop{\downarrow}\nolimits}
\newcommand\uuarrow{\rlap{$\uparrow$}\raise.5ex\hbox{$\uparrow$}}
\newcommand\ddarrow{\rlap{$\downarrow$}\raise.5ex\hbox{$\downarrow$}}
\newcommand\suf{\mathop{\text{suf}}}
\newcommand\pref{\mathop{\text{pref}}}

\newcommand\Img{\mathop{\mathrm{Im}}\nolimits}

\newcommand\Eval[1]{\left\llbracket{#1}\right\rrbracket}
\newcommand\SEval[1]{\Eval{#1}} 
\newcommand\meet{\mathrel{\wedge}}
\newcommand\topcap{\tau}

\newcommand\infy[1]{(\infty) {#1}}
\newcommand\eqdef{\mathrel{\buildrel \text{def}\over=}}
\newcommand\diff{\smallsetminus}
\newcommand\identity[1]{\mathrm{id}_{#1}}

\begin{document}


\baselineskip=17pt


\title{Infinitary Noetherian Constructions I. Infinite Words}

\author{Jean Goubault-Larrecq\\
  Universit\'e Paris-Saclay, ENS Paris-Saclay, CNRS, LSV\\
  91190, Gif-sur-Yvette, France\\
  E-mail: \url{goubault@lsv.fr}}
\thanks{The author was supported by grant ANR-17-CE40-0028 of the
  French National Research Agency ANR (project BRAVAS)}

\date{}


\begin{abstract}
  We define and study Noetherian topologies for spaces of infinite
  sets and for spaces of infinite words.  In each case, we also obtain
  S-representations, namely, computable presentations of the
  sobrifications of those spaces.
\end{abstract}

\subjclass[2020]{Primary 54G99; 
  Secondary 06A07
  }

\keywords{Noetherian spaces, well-quasi-orders, infinite words.}

\maketitle
\markboth{J. Goubault-Larrecq}{Infinite Words}

\section{Introduction}
\label{sec:intro}

A quasi-ordering, a.k.a.\ a preordering, is a reflexive and transitive relation.
A \emph{well-quasi-order} (\emph{wqo}) is a quasi-ordered set in which
every infinite sequence ${(x_n)}_{n \in \nat}$ must contain two
elements $x_m \leq x_n$ with $m < n$.  Well-quasi-orders are a
fundamental tool in mathematics and computer science, however they are
not closed under several infinitary constructions; e.g., the set of
all subsets of a wqo is not in general wqo \cite{Rado:wqo}, and a
similar problem plagues sets of infinite words, and of infinite trees,
over a well-quasi-ordered alphabet.  Nash-Williams discovered that a
strengthening of the notion of wqo, the notion of \emph{better
  quasi-orders} (\emph{bqo}), was closed under the usual finitary
constructions that preserved being wqo (finite words, finite trees,
etc.), and also under their infinitary variants \cite{NW:bqo}.

A \emph{Noetherian space} is a topological space in which every open
set is compact, i.e., in which every open cover of an open set
contains a finite subcover---we do not assume any separation axiom
here.  It was observed in \cite{Gou-lics07} that Noetherian spaces
formed a natural topological generalization of the order-theoretic
notion of wqo.  Noetherian spaces are closed under the same finitary
constructions as wqos (finite words under embedding, finite trees
under homeomorphic embedding, etc., see
\cite[Section~9.7]{JGL-topology}), but also under some infinitary
constructions.  In \cite{Gou-lics07}, notably, we remarked that the
so-called Hoare powerdomain of a Noetherian space---equivalently, its
powerset under the so-called lower Vietoris topology---is Noetherian.

The main purpose of this paper is to show that Noetherianness is
preserved under some of the usual infinitary constructions that
spurred the invention of bqos.

A secondary purpose is to design those
infinitary constructions in such a way that the closed subsets have
finite representations suitable for an implementation on a computer.
We do not mean to develop this here, but one should note that
infinite words are pervasive in verification, mostly as infinite runs
in various forms of automata, such as B\"uchi automata, see
\cite{Thomas:aut:inf} for a survey.  Downwards-closed subsets of
runs, a very closely related notion, were instrumental in the
pioneering paper by Leroux and Schmitz on the complexity of VASS
reachability \cite{DBLP:conf/lics/LerouxS15}, too.

Let us illustrate our goal by an example of a finitary construction,
taken from \cite[Section~9.7]{JGL-topology} and
\cite[Section~7]{FGL:compl}.  Let $X^*$ denote the set of finite words
over an alphabet $X$ (not necessarily finite).  For every
quasi-ordering $\leq$ on $X$, the (scattered) word embedding
quasi-ordering $\leq^*$ on $X^*$ is defined by $w \leq^* w'$ if and
only if $w'$ can be obtained from $w$ by increasing some letters from
$w$ and by inserting arbitrarily many new letters at arbitrary
positions.  Higman's Lemma \cite{Higman:Lemma} states that $\leq^*$ is
a well-quasi-ordering if and only if $\leq$ is.  Similarly, the
\emph{word topology} on $X^*$, where $X$ is now a topological space,
is generated by basic open sets of the form
$\langle U_1; U_2; \cdots ; U_n \rangle \eqdef X^* U_1 X^* U_2 \cdots
X^* U_n X^*$, where $n \in \nat$ and each $U_i$ is open in $X$---those
are the sets of words that contain a letter in $U_1$, then a letter in
$U_2$ to the right of the previous one, and so on, until we find a
letter in $U_n$.  (Note that they form a base, not just a subbase.)
Then the following hold (all required notions will be introduced in
Section~\ref{sec:preliminaries}):
\begin{enumerate}[label=(\Alph*)]
\item $X$ is Noetherian if and only if $X^*$ (with the word topology)
  is Noetherian.
\item The specialization preordering of $X^*$ is $\leq^*$, where
  $\leq$ denotes the specialization preordering of $X$.
\item If $X$ is wqo, then so is $X^*$.
\item If $X$ has an S-representation (a certain, computable, way of
  representing the irreducible closed subsets of $X$, and therefore
  all closed subsets of $X$), then $X^*$ has an S-representation, too.
\end{enumerate}
We wish to obtain similar results for infinitary constructions, e.g.,
spaces of infinite words.  Our proposals will allow us to obtain
equivalents of (A) and (D).  (B) will only hold if $X$ is wqo, and (C)
will hold if and only if $X$ is essentially finite (see below);
however negative the latter result seems, one should note that we
define a topological space as \emph{wqo} if and only if its
specialization preordering is a wqo \emph{and} its topology is
Alexandroff---that is a pretty strong requirement.


\paragraph{\emph{Outline.}}
After some preliminaries in Section~\ref{sec:preliminaries}, we
examine the case of the powerset $\pow (X)$ for Noetherian $X$.  That
$\pow (X)$ is Noetherian in that case is not new \cite{Gou-lics07,
  JGL-topology,FGL:compl}, but it seems important to understand why.
This will occupy Section~\ref{sec:pow:AD}, in which we will deal with
properties (A) and (D) in that case.  In Section~\ref{sec:pow:BC}, we
examine properties (B) and (C).  That is new.  As promised, property
(B) will hold only when $X$ is wqo, and (C) only when $X$ is
essentially finite We then make a small detour and introduce a few
useful results pertaining to so-called initial maps in
Section~\ref{sec:initial-maps}.  With all that in our hands, we will
proceed to show that the space $X^\omega$ of all infinite words over
$X$, with a natural topology, enjoys properties (A) through (D)---in
the case of (B) and (C), exactly with the same restrictions on $X$ as
above.

\section{Preliminaries}
\label{sec:preliminaries}

We have already defined well-quasi-orders.  They can be defined in
many equivalent ways.  Notably, as already observed by Higman
\cite[Theorem~2.1]{Higman:Lemma}, it is equivalent to require any of
the following properties, for a quasi-ordered set $X$: (i) every
upwards-closed subset of $X$ is the upward-closure of a finite set;
(ii) the lattice of upwards-closed subsets of $X$ has the ascending
chain condition, namely: every ascending sequence
${(U_n)}_{n \in \nat}$ of upwards-closed subsets is stationary, in
other words, there is a rank $n_0$ such that $U_n = U_{n_0}$ for every
$n \geq n_0$ (in general, we say that a quasi-ordering $\leq$ has the
\emph{ascending chain condition} if and only if it has no strictly
ascending infinite sequence $x_0 < x_1 < \cdots < x_n < \cdots$, where
$x < y$ means $x \leq y$ and $y \not\leq x$); (iv) every infinite
sequence of elements of $X$ has an infinite ascending subsequence; (v)
$X$ is wqo.  (We omit characterizations (iii) and (vi), which we will
not require.)

Let us turn to topology, for which we refer the reader to
\cite{JGL-topology}.  Noetherian spaces are specifically covered in
Section~9.7 there.  Note that none of the topologies we will consider
are Hausdorff.  In fact, a Hausdorff topological space is Noetherian
if and only if it is finite.

A \emph{subbase} of a topology is any family of open sets that
generates the family.  A \emph{base} of a topology is a family of open
sets such that every open set can be written as a union of basic open
sets.  We write $cl (A)$ for the closure of a subset $A$ of a
topological space.  We will often use the fact that $cl (A)$
intersects an open set $U$ if and only if $A$ intersects $U$.

Noetherian spaces have many equivalent characterizations (compare with
the equivalent characterizations (i)--(vi) of wqos mentioned earlier).
Those are also the spaces in which every ascending sequence
${(U_n)}_{n \in \nat}$ of open subsets is stationary; or also the
spaces in which every descending sequence ${(C_n)}_{n \in \nat}$ of
closed subsets is stationary.  The first of those characterizations
shows that Noetherianness is a property that depends only on the
lattice of open subsets of the space, not on its point.

Noetherian spaces are closed under finite products, finite coproducts,
subspaces, under the process of replacing the topology by a coarser
one, under images by continuous maps, and various other constructions,
such as the $X^*$ construction.

Every topological space has a \emph{specialization} preordering
$\leq$, defined by $x \leq y$ if and only if every open neighborhood
of $x$ contains $y$.  We then say that $x$ is less than or equal to
$y$, or \emph{below} $y$, or that $y$ is larger than or equal to $x$,
or \emph{above} $x$.  The closure of $\{x\}$ is the \emph{principal
  ideal} $\dc x$, namely the set of all points below $x$ in that
quasi-ordering.  (Symmetrically, we write $\upc x$ for the set of all
points above $x$.)  An \emph{Alexandroff topology} is a topology in
which every intersection of open subsets is open, or equivalently, in
which the open subsets are exactly the upwards-closed subsets in the
specialization preordering $\leq$.  The Alexandroff topology of a
given quasi-ordering $\leq$ is, correspondingly, the collection of all
its upwards-closed sets.  Among the topologies with a given
specialization preordering $\leq$, the Alexandroff topology is the
finest, and the coarsest is the \emph{upper topology}, whose closed
sets are intersections of sets of the form $\dc E$, $E$ finite; the
notation $\dc E$ denotes $\bigcup_{x \in E} \dc x$.

There is some degree of ambiguity in a notation such as $\dc E$, which
will be particularly apparent when we work in spaces of subsets.  For
$E \in \pow (X)$, where $\pow (X)$ is equipped with the inclusion
ordering, say, $\dc E$ might denote $\{E' \in \pow (X) \mid E'
\subseteq E\}$ or $\{x \in X \mid \exists y \in E, x \leq y\}$.  In
such cases, we will disambiguate by writing $\dc_{\pow (X)} E$ for the
first set, and $\dc_X E$ for the second one.

It turns out that a quasi-ordering $\leq$ is a well-quasi-ordering if
and only if its Alexandroff topology is Noetherian
\cite[Proposition~9.7.17]{JGL-topology}.
For short, we will say that a topological space is a \emph{wqo} if and
only if it is Noetherian and its topology is the Alexandroff topology,
equivalently if and only if its topology is the Alexandroff topology
of a well-quasi-ordering.

A subset $C$ of a topological space $X$ is \emph{irreducible} if and
only if it is non-empty, and for all closed subsets $C_1$, $C_2$ of
$X$ such that $C \subseteq C_1 \cup C_2$, we have $C \subseteq C_1$ or
$C \subseteq C_2$.  Equivalently: $C$ is non-empty, and for all open
subsets $U_1$, $U_2$ of $X$ that intersect $C$, $U_1 \cap U_2$ also
intersects $C$.

A \emph{sober space} is a topological space in which every irreducible
closed subset is the closure $cl (\{x\}) = \dc x$ of a unique point
$x$.  (Chapter~8 of \cite{JGL-topology} is all about sober spaces.)
The (standard) \emph{sobrification} $\Sober X$ of a topological space
$X$ is its set of irreducible closed subsets, with the
\emph{hull-kernel topology}, whose open subsets are (exactly) the sets
of the form
$\diamond U \eqdef \{C \in \Sober X \mid C \cap U \neq \emptyset\}$,
where $U$ ranges over the open subsets of $X$.  The specialization
(quasi-)ordering of $\Sober X$ is inclusion.  $\Sober X$ is always
sober, there is a continuous map
$\eta_X \colon X \to \Sober X \colon x \mapsto \dc x$, and for every
continuous map $f \colon X \to Y$ where $Y$ is sober, there is a
unique continuous map $\hat f \colon \Sober X \to Y$ such that
$\hat f \circ \eta_X = f$.

$\Sober$ defines a endofunctor on the category of topological spaces,
and its action on morphisms is defined by $\Sober (f) (C) \eqdef cl (f
[C])$, where $f [C]$ denotes the image of $C$ under $f$.  In
particular, $cl (f [C])$ is irreducible closed for every irreducible
closed set $C$ and every continuous map $f$.

Sober spaces are closed under arbitrary topological products.
Furthermore, the sobrification of any product of spaces is
homeomorphic to the product of the sobrifications.  Explicitly, and in
the binary case, given any two irreducible closed subsets $C$ of $X$
and $C'$ of $Y$, $C \times C'$ is irreducible closed in $X \times Y$.
Moreover, all irreducible closed subsets of $X \times Y$ are of this
form: $(C, C') \mapsto C \times C'$ is the indicated homeomorphism
from $\Sober (X) \times \Sober (Y)$ to $\Sober (X \times Y)$.

A space is Noetherian if and only if its sobrification is Noetherian.
Indeed, the map $U \mapsto \diamond U$ is an order-isomorphism, hence
the lattice of open sets of $X$ has the ascending chain condition if
and only if the lattice of open sets of $\Sober X$ has it as well.

We say that a quasi-ordered set (resp., a topological space) is
\emph{essentially finite} if and only if it has only finitely many
upwards-closed subsets (resp., open subsets).  Note that the topology
of an essentially finite topological space is Alexandroff, and
trivially Noetherian.  A topological space $X$ is essentially finite
if and only if its $T_0$ quotient, namely the quotient $X / \equiv$
where $\equiv \eqdef \leq \cap \geq$, is finite.

The sober Noetherian spaces are particularly interesting, as they can
be characterized entirely in terms of their specialization
preordering.  Explicitly, the sober Noetherian spaces are exactly
the sets $X$ with a well-founded quasi-ordering $\leq$ such that every
finite intersection of principal ideals can be expressed as a finite
union of principal ideals (a quasi-ordering $\leq$ is
\emph{well-founded} if and only if every strictly descending chain is
finite); furthermore, the topology of $X$ is uniquely determined as
the upper topology of $\leq$.  In that case, the closed subsets are
exactly the sets of the form $\dc E$ with $E$ \emph{finite}, which
makes them suitable for a representation on a computer---provided all
the elements of $E$ are themselves representable.

As a corollary, the closed subsets $C$ of a Noetherian space $X$ are
exactly the finite unions of irreducible closed subsets $C_1$,
\ldots, $C_n$ of $X$.  Indeed, given any closed subset $C$ of $X$, the
set $\dc_{\Sober X} C$ of all irreducible closed subsets of $X$ below
(included in) $C$ is equal to $\Sober X \diff \diamond (X \diff C)$,
hence is closed in $\Sober X$.  Also, $\eta_X^{-1} (\dc \{C\}) = C$.
Since $\Sober X$ is Noetherian, one can write $\dc \{C\}$ as
$\dc \{C_1, \cdots, C_n\}$ for finitely many points $C_1$, \ldots,
$C_n$ of $\Sober X$, and then $C = \eta_X^{-1} (\dc \{C\})$ is the
union of the finitely many irreducible closed subsets
$\eta_X^{-1} (\dc \{C_i\}) = C_i$, $1\leq i\leq n$.

We will be interested in computer representations of irreducible
closed subsets of $X$ (i.e., of elements of $\Sober (X)$), and this
will immediately allow us to represent all closed subsets $C$ as
finite sets $\{C_1, \cdots, C_n\}$, where each $C_i$ is in
$\Sober (X)$.  If we can decide inclusion of irreducible closed
subsets, one can also decide the inclusion of arbitrary closed
subsets: if $C$ is represented by the finite set
$\{C_1, \cdots, C_m\}$ and $C'$ is represented by the finite set
$\{C'_1, \cdots, C'_n\}$, then $C \subseteq C'$ if and only if for
every $i$, there is a $j$ such that $C_i \subseteq C'_j$.  This is a
simple consequence of the fact that each $C_i$ is irreducible.  We
will also require to be able to compute the intersection $C \cap C'$
of any two irreducible closed subsets of $X$ as a finite union
$C_1 \cup \cdots \cup C_n$ of irreducible closed subsets.

Those computability requirements are formalized by the notion of an
\emph{S-representation} \cite[Definition~5.1]{FGL:compl}.  An
S-representation of a Noetherian space $X$ is a 5-tuple
$(S, \SEval \_, \unlhd, \topcap, \meet)$ where:
\begin{enumerate} 
\item $S$ is a recursively enumerable set of so-called \emph{codes}
  (of irreducible closed subsets);
\item $\SEval \_$ is a surjective map from $S$ to $\Sober X$;
\item $\unlhd$ is a decidable relation such that, for all codes
  $a, b\in S$, $a \unlhd b$ iff $\SEval a \leq \SEval b$;
\item $\topcap$ is a finite subset of $S$, such that $X = \bigcup_{a
    \in \topcap} \SEval a$;
\item $\meet$ is a computable map (the \emph{intersection map}) from
  $S \times S$ to the collection $\Pfin (S)$ of finite subsets of $S$
  (and we write $a \meet b$ for $\meet (a,b)$) such that
  $\SEval a \cap \SEval b = \bigcup_{c \in a \meet b} \SEval c$.
\end{enumerate}
Let us take $X^*$, with the word topology, as an example.  We use
standard notations for certain regular languages on $X$: for every
$C \subseteq X$, $C^?$ denotes the set of words of at most one letter,
and that letter is in $C$; for every $F \subseteq X$, $F^*$ is the set
of words whose letters are all in $F$; for all $A, B \subseteq X^*$,
$AB$ denotes the set of all concatenations of one word from $A$ and
one from $B$; $\epsilon$ denotes both the empty word and the language
$\{\epsilon\}$.  A \emph{word product} is a language of the form
$P \eqdef A_1 A_2 \cdots A_N$, where each $A_i$ is an \emph{atom},
i.e., a language of the form $C^?$ with $C \in \Sober X$ or $F^*$
where $F$ is a closed subset of $X$.  When $X$ is Noetherian, the
irreducible closed subsets of $X^*$ are exactly the word products
\cite[Proposition~7.14]{FGL:compl}.  One can also decide inclusion of
word products in polynomial time with an oracle deciding inclusion in
$\Sober X$ \cite[Lemma~7.10, Corollary~7.11]{FGL:compl}, and compute
intersections of word products as finite unions of word products in
polynomial time with an oracle computing binary intersections in
$\Sober X$ as finite unions of irreducible closed subsets
\cite[Lemma~7.13]{FGL:compl}.  Formally:
\begin{proposition}[Theorem~7.15 of \cite{FGL:compl}]
  \label{prop:srepr:X*}
  Given an S-representation
  $(S, \SEval \_, \unlhd, \allowbreak \topcap, \meet)$ of a Noetherian
  space $X$, the following tuple
  $(S', \SEval \_', \unlhd', \topcap', \meet')$ is an S-representation
  of $X^*$:
  \begin{enumerate}
  \item $S'$ is the collection of all (syntactic) word products over
    the alphabet $S$, namely all regular expressions
    $A_1 A_2 \cdots A_N$ where each $A_i$ is either an expression of
    the form $a^?$ with $a \in S$, or $u^*$ where $u$ is a finite
    subset of $S$ (we write $\varepsilon$ when $N=0$).
  \item
    $\SEval {A_1 A_2 \cdots A_N}' \eqdef \SEval {A_1}' \SEval {A_2}'
    \cdots \SEval {A_N}'$, where we let
    $\SEval {a^?}' \eqdef \SEval a^?$ and
    $\SEval {\{a_1, \cdots, a_n\}^*}' \eqdef (\SEval {a_1} \cup \cdots
    \cup \Eval {a_n})^*$.
  \item $\unlhd'$ is defined inductively by:
    \begin{align*}
      \varepsilon \unlhd' Q & \text{ is always true} \\
      P \unlhd' \varepsilon & \text{ is false, if }P \neq \varepsilon \\
      a^? P \unlhd' {b}^? Q & \text{ iff }
                                \left\{
                                \begin{array}{ll}
                                  P \unlhd' Q & \text{if } a \unlhd b \\
                                  a^? P \unlhd' Q & \text{otherwise}
                                \end{array}
                                                     \right. \\
      a^? P \unlhd' {v}^* Q
                             & \text{ iff }
                               \left\{
                               \begin{array}{ll}
                                 P \unlhd' {v}^* Q
                                 & \text{if }\exists b \in v, a \unlhd b \\
                                 a^? P \unlhd' Q & \text{otherwise}
                               \end{array}
                                                    \right. \\
      u^* P \unlhd' {b}^? Q
                             & \text{ iff }
                               \left\{
                               \begin{array}{ll}
                                 P \unlhd' {b}^? Q& \text{if
                                                      }u=\emptyset \\
                                 u^* P \unlhd' Q & \text{otherwise}
                               \end{array}
                                                                         \right. \\
      u^* P \unlhd' {v}^* Q
                             & \text{ iff }
                               \left\{
                               \begin{array}{ll}
                                 P \unlhd' {v}^* Q
                                 & \text{if }\forall a \in u, \exists
                                   b \in v, a \unlhd b \\
                                 u^* P \unlhd' Q & \text{otherwise}
                               \end{array}
                                                                         \right.
    \end{align*}
  \item $\topcap'$ is $\{\topcap^*\}$.
  \item $\meet'$ is implemented by the following clauses (together
    with the obvious symmetric clauses):
    \begin{align}
      \label{eq:inter:eps}
      \varepsilon \meet' Q & \eqdef \{\varepsilon\} \\
      \label{eq:inter:??}
      a^? P \meet' {b}^? Q & \eqdef (a^? P \meet' Q) \cup (P \meet'
                               {b}^? Q) \\
      \nonumber
                            & \;\cup \{{c}^? R \mid c \in a \meet b, R \in
                              P \meet' Q\} \\
      \label{eq:inter:?*}
      a^? P \meet' {v}^* Q
                           & \eqdef
                             \left\{
                             \begin{array}{l}
                               \{{c}^? R
                               \mid c \in \bigcup_{b \in v} (a
                               \meet b),
                               R \in P \meet' {v}^* Q\}
                               \cup (a^? P \meet' Q)\\
                               \qquad\qquad
                               \qquad \text{if } a \meet b \neq
                               \emptyset \text{ for some }b \in v, \\
                               (P \meet' {v}^* Q) \cup (a^? P \meet' Q)
                               \qquad \text{otherwise}
                             \end{array}
      \right. \\
      \label{eq:inter:**}
      u^* P \meet' {v}^* Q
                            & \eqdef \{(\bigcup_{a \in u, b \in v} a \meet
                              b)^* R \mid R \in (P \meet' {v}^* Q) \cup (u^* P \meet' Q)\}.
    \end{align}
  \end{enumerate}
\end{proposition}

\begin{remark}
  \label{rem:meet'}
  One can optimize the procedures above in a number of ways.  In the
  definition of $\meet'$, one can remove any subsumed word product in
  the result.  A word product $P$ is \emph{subsumed} by another one,
  $Q$, in a given set, if and only if $P \unlhd' Q$, or equivalently
  $\SEval {P}' \subseteq \SEval {Q}'$.  As a special case, in
  (\ref{eq:inter:?*}), if $Q = \varepsilon$, then we can remove
  $a^? P \meet' Q$ ($=\{\varepsilon\}$), which is subsumed by some
  other word product, since $\SEval {\varepsilon}' = \{\epsilon\}$ is
  included in the denotation of the remaining word products (the union
  of the sets $\SEval {{c}^? R}'$ where
  $c \in \bigcup_{b \in v} (a \meet b)$ and
  $R \in P \meet' {v}^* Q$ if $a \meet b \neq \emptyset$ for some
  $b \in v$, the union of the sets $\SEval {R}'$ where
  $R \in P \meet' {v}^* Q$ otherwise).
\end{remark}

\section{Powersets}
\label{sec:powersets}

\subsection{Properties (A) and (D)}
\label{sec:pow:AD}

Let $\pow (X)$ denote the powerset of a space $X$, with the
\emph{lower Vietoris topology}, generated by subbasic open sets of the
form
$\Diamond U \eqdef \{A \in \pow (X) \mid A \cap U \neq \emptyset\}$.
By that, we mean that the open subsets of $\pow (X)$ are the unions of
finite intersections $\bigcap_{i=1}^n \Diamond U_i$.  (Note the
similarity of that notation with the open subsets $\diamond U$ of
$\Sober X$.  They are defined the same way, but the sets $\Diamond U$
only form a subbase of the lower Vietoris topology, whereas the sets
$\diamond U$ range over all the open sets in the hull-kernel topology
on $\Sober X$.  Also, $\Diamond$ and $\diamond$ commute with arbitrary
unions, but $\diamond$ additionally commutes with finite
intersections.)

The subset of $\pow (X)$ consisting of all closed subsets of $X$ is
called the \emph{Hoare powerspace} of $X$, and will be written as
$\Hoare (X)$.  We again write $\Diamond U$ for the open set
$\{C \in \Hoare (X) \mid C \cap U \neq \emptyset\}$.
Those sets generate the subspace topology on $\Hoare (X)$, and we will
also call it the lower Vietoris topology.  For any set $A$, $A$
intersects an open set $U$ if and only if $cl (A)$ intersects $U$, and
this implies that the function that maps every open subset of
$\pow (X)$ to its intersection with $\Hoare (X)$ is an
order-isomorphism.  The following lemma, which is of independent
interest, shows that $\Hoare (X)$ is homeomorphic to
$\Sober (\pow (X))$.  We also deal with $\pow^* (X)$, the subspace of
\emph{non-empty} subsets of $X$, and with $\Hoare^* (X)$, the subspace
of non-empty closed subsets of $X$.
\begin{lemma}[Lemma~5.10 of \cite{FGL:compl}]
  \label{lemma:H:sobr}
  The map $F \mapsto \dc F$ is a homeomorphism from $\Hoare (X)$ onto
  $\Sober (\pow (X))$, resp.\ from $\Hoare^* (X)$ onto
  $\Sober (\pow^* (X))$.  \qed
\end{lemma}
It follows that for every space $X$, $\pow (X)$ is Noetherian if and
only if $\Hoare (X)$ is Noetherian, and similarly for $\pow^* (X)$ and
$\Hoare^* (X)$.  It is easy to see that every subspace and every
homeomorph of a Noetherian space is Noetherian, so any of those
properties implies that $X$, which is homeomorphic to the subspace of
points $\{x\}$ in $\pow (X)$ (resp., $\pow^* (X)$), is Noetherian.
Conversely, if $X$ is Noetherian, then $\subseteq$ is well-founded on
$\Hoare (X)$.  Any finite intersection of principal ideals
$\dc_{\Hoare (X)} F_i$, $1\leq i\leq n$, in $\Hoare (X)$ can be
expressed as a finite union of principal ideals, in fact just as
$\dc_{\Hoare (X)} (F_1 \cap \cdots \cap F_n)$.  It follows that
$\Hoare (X)$ is Noetherian, and sober, with the upper topology of
inclusion.  Since the complement of
$\dc_{\Hoare (X)} \{F_1, \cdots, F_n\}$ is equal to
$\Diamond U_1 \cap \cdots \cap \Diamond U_n$, where each $U_i$ is the
complement of $F_i$ in $X$, that upper topology is none other than the
lower Vietoris topology.

The next proposition follows easily, and is a reformulation of
Theorem~5.11, (A)--(C), of \cite{FGL:compl}; that theorem actually
gives a full description of an S-representation for $\pow (X)$ and for
$\pow^* (X)$, while Theorem~5.8 of \cite{FGL:compl} gives the
analogous S-representation for $\Hoare (X)$ and for $\Hoare^* (X)$.
\begin{proposition}
  \label{prop:pow:Noeth}
  For every topological space $X$, $X$ is Noetherian if and only if
  $\pow (X)$ (resp., $\Hoare (X)$, $\pow^* (X)$, $\Hoare^* (X)$) is.

  Letting $Y \eqdef \pow (X)$ (resp., $\pow^* (X)$), the irreducible
  closed subsets of $Y$ are exactly the sets of the form
  $\dc_Y F = \{A \in Y \mid A \subseteq F\}$, where $F \in \Hoare (X)$
  (resp., $\Hoare^* (X)$).

  In particular, if $X$ is Noetherian, then the irreducible closed
  subsets of $Y$ can be represented as finite sets
  $\{C_1, \cdots, C_n\}$ (resp., with $n \geq 1$), denoting
  $\dc_Y (C_1 \cup \cdots \cup C_n)$, where each $C_i \in \Sober X$;
  if inclusion is decidable on $\Sober X$, then inclusion is decidable
  on $\Sober Y$: if $F$ is represented by the finite set
  $\{C_1, \cdots, C_m\}$ and $F'$ is represented by the finite set
  $\{C'_1, \cdots, C'_n\}$, then $F \subseteq F'$ if and only if for
  every $i$, there is a $j$ such that $C_i \subseteq C'_j$.  \qed
\end{proposition}
Those match properties (A) and (D) mentioned in the introduction, as
promised.

\subsection{Properties (B) and (C)}
\label{sec:pow:BC}

As for property (B), the specialization preordering on $\pow (X)$
(resp., $\pow^* (X)$) is given by $A \leq^\flat B$ if and only if
$cl (A) \subseteq cl (B)$.  When $X$ is a wqo, $cl (A) = \dc_X A$, so
$A \leq^\flat B$ if and only if for every $a \in A$, there is a
$b \in B$ such that $a \leq b$, and we retrieve the usual
\emph{domination} (a.k.a., \emph{Hoare}) quasi-ordering.

We now inquire about property (C).  One may wonder when $\pow (X)$ is
wqo, in the sense that its topology is both Alexandroff and
Noetherian.  One might think that this would be the case if and only
if $X$ is $\omega^2$-wqo (see \cite{Laver:2wqo,Marcone:bqo} or
\cite{Jancar:wqo:pow} for example).  This is wrong, as we will see in
Proposition~\ref{prop:pow:wqo} below.  If $X$ is $\omega^2$-wqo, what
we obtain is that the domination quasi-ordering on $\pow (X)$ is a
well-quasi-ordering (this can be taken as a definition of an
$\omega^2$-wqo), \emph{not} that the lower Vietoris topology is
Alexandroff.  We will use the following lemma.
\begin{lemma}
  \label{lemma:wqo:acc}
  \begin{enumerate}
  \item A topological space whose lattice of open subsets is
    well-founded under inclusion has the Alexandroff topology of its
    specialization preordering, and that quasi-ordering has the
    ascending chain condition.
  \item A well-quasi-ordering with the ascending chain condition is
    essentially finite.
  \item A Noetherian space whose lattice of open subsets is
    well-founded is essentially finite.
  \end{enumerate}
\end{lemma}
\begin{proof}
  (1) Let us assume that the lattice of open subsets of $X$ is
  well-founded.  Given any $x \in X$, there is a minimal open
  neighborhood $U_x$ of $x$.  By definition of the specialization
  preordering $\leq$, for every point $y \in X$ such that
  $x \not\leq y$, there is an open subset $U$ of $X$ that contains $x$
  but not $y$.  Since $U \cap U_x$ is an open neighborhood of $x$, the
  minimality of $U_x$ entails that $U \cap U_x = U_x$, that is,
  $U_x \subseteq U$.  It follows that $y$ is not in $U_x$.  We have
  shown the implication $x \not\leq y \limp y \not\in U_x$, from which
  we deduce $U_x \subseteq \upc x$.  Every open set is upwards-closed
  in the specialization preordering, so $U_x = \upc x$.

  From this, we deduce that $\upc x$ is open for every $x \in X$.
  Every upwards-closed subset $A$ is equal to $\bigcup_{x \in A} \upc
  x$, hence is open.  Hence the topology of $X$ is the Alexandroff
  topology of $\leq$.

  We now consider any strictly increasing sequence
  $x_0 < x_1 < \cdots < x_n < \cdots$.  Then the sets $\upc x_n$ form
  a strictly descending sequence of open subsets, contradicting our
  well-foundedness assumption.  Hence $\leq$ has the ascending chain
  condition.

  (2) By contradiction, let us assume that there is an infinite set
  $A$ whose elements are pairwise inequivalent with respect to
  $\equiv \eqdef \leq \cap \geq$.  We extract a countable infinite
  subset ${(x_n)}_{n \in \nat}$ of $A$.  In a well-quasi-ordering,
  every infinite sequence has an infinite ascending subsequence, so we
  may assume without loss of generality that
  $x_0 \leq x_1 \leq \cdots \leq x_n \leq \cdots$.  By the ascending
  chain condition, only finitely many of those inequalities can be strict,
  hence $x_n \equiv x_{n+1}$ for at least one $n$.  That is absurd.

  (3) If $X$ is Noetherian with a well-founded lattice of open
  subsets, by (1) it is Alexandroff, hence wqo, and satisfies the
  ascending chain condition.  We conclude by (2).
\end{proof}

\begin{proposition}
  \label{prop:pow:wqo}
  Let $X$ be a Noetherian space.  The lower Vietoris topology on
  $\pow (X)$ (resp., $\Hoare (X)$, $\pow^* (X)$, $\Hoare^* (X)$) is
  Alexandroff if and only if $X$ is essentially finite.
\end{proposition}
\begin{proof}
  The if direction is clear.  Let $Y$ be $\pow (X)$ (resp.,
  $\Hoare (X)$, $\pow^* (X)$, $\Hoare^* (X)$), and let us assume that
  $Y$ is Alexandroff.  We use Lemma~\ref{lemma:wqo:acc}~(3), first
  showing that the lattice of open subsets of $X$ is well-founded, or
  equivalently that there is no infinite strictly ascending sequence
  of closed subsets of $X$.

  For the sake of contradiction, we assume that there is such an
  infinite strictly ascending sequence
  $F_0 \subsetneq F_1 \subsetneq \cdots \subsetneq F_n \subsetneq
  \cdots$.  Up to the removal of $F_0$, we may assume that every $F_n$
  is non-empty: this is needed for the cases where $Y$ is $\pow^* (X)$
  or $\Hoare^* (X)$.  For each $n \in \nat$,
  $\dc_Y F_n = Y \diff \Diamond (X \diff F_n)$ is closed.  Since $Y$
  is Alexandroff, any union of closed subsets of $Y$ is closed, so
  $\mathcal F \eqdef \bigcup_{n \in \nat} \dc_Y F_n$ is closed.

  Let $F_\infty$ be the closure of $\bigcup_{n \in \nat} F_n$ in $X$.
  We claim that $F_\infty$ is in $\mathcal F$.  Otherwise, by the
  definition of the lower Vietoris topology, $F_\infty$ would be in
  some finite intersection $\bigcap_{i=1}^m \Diamond U_i$, disjoint
  from $\mathcal F$, where each $U_i$ is open in $X$.  For each $i$,
  $U_i$ would then intersect $F_\infty$, hence
  $\bigcup_{n \in \nat} F_n$, hence $F_{n_i}$ for some $n_i \in \nat$.
  Let $n \in \nat$ be larger than every $n_i$, $1\leq i\leq m$.  Then
  $F_n$ intersects every $U_i$, $1\leq i\leq m$, as well, so
  $F_n \in \bigcap_{i=1}^m \Diamond U_i$.  However, $F_n$ is in
  $\mathcal F$, by definition of $\mathcal F$, which is impossible
  since $\mathcal F$ is disjoint from $\bigcap_{i=1}^m \Diamond U_i$
  by assumption.

  Since $F_\infty$ is in $\mathcal F$, it is in some $\dc_Y F_n$.  In
  particular, $F_{n+1} \subseteq F_\infty \subseteq F_n$, which is
  impossible.  Hence there is no strictly ascending sequence of closed
  subsets of $X$, and we conclude by Lemma~\ref{lemma:wqo:acc}~(3).
\end{proof}

\section{Initial maps}
\label{sec:initial-maps}

We will use the following additional facts about Noetherian spaces.
An \emph{initial} map $f \colon Y \to Z$ between topological spaces is
one such that the open subsets of $Y$ are exactly the sets of the form
$f^{-1} (W)$, $W$ open in $Z$.  All initial maps are continuous.

\begin{remark}
  \label{rem:initial}
  Given a subbase of the topology of $Y$, a practical way of checking
  that $f \colon Y \to Z$ is initial consists in verifying that $f$ is
  continuous, and that every subbasic open subset $V$ of $Y$ can be
  written as $f^{-1} (W)$ for some open subset $W$ of $Z$.  In other
  words, we do not need to check the latter for every open subset of
  $Y$, just for subbasic open sets.  Indeed, every open subset $V$ of
  $Y$ can be written as $\bigcup_{i \in I} \bigcap_{j=1}^{n_i} V_{ij}$
  where each $V_{ij}$ is subbasic, and if we can write each $V_{ij}$
  as $f^{-1} (W_{ij})$ with $W_{ij}$ open in $Z$, then $V$ is equal to
  $f^{-1} (\bigcup_{i \in I} \bigcap_{j=1}^{n_i} W_{ij})$.
\end{remark}

A general way of finding initial maps is as follows.  Let $Z$ be a
topological space and $f$ be a map from a set $Y$ to $Z$.  With the
coarsest topology on $Y$ that makes $f$ continuous, $f$ is initial.
This is notably the case of topological embeddings, which are those
initial maps that are injective.

\begin{lemma}
  \label{lemma:initial:noeth}
  Let $f \colon Y \to Z$ be an initial map between topological
  spaces.  If $Z$ is Noetherian, then $Y$ is Noetherian.
\end{lemma}
\begin{proof}
  The open subsets of $Y$ are the sets $f^{-1} (W)$, $W$ open in $Z$.
  Let ${(f^{-1} (W_n))}_{n \in \nat}$ be a monotonic sequence of open
  subsets in $Y$.  Replacing $W_n$ by
  $W_0 \cup W_1 \cup \cdots \cup W_n$, we may assume that
  ${(W_n)}_{n \in \nat}$ is also a monotonic sequence.  Since $Z$ is
  Noetherian, all sets $W_n$ are equal for $n$ large enough.  Hence
  all sets $f^{-1} (W_n)$ are equal for $n$ large enough.
\end{proof}

\begin{lemma}
  \label{lemma:irred:f-1}
  Let $f \colon Y \to Z$ be an initial map.  The irreducible closed
  subsets $D$ of $Y$ are all of the form $f^{-1} (C)$ where $C$ is
  some irreducible closed in $Z$.  More precisely, one can always
  choose $C \eqdef cl (f [D])$.
\end{lemma}
\begin{proof}
  Let $D$ be irreducible closed in $Y$, and consider $C \eqdef cl (f
  [D])$.  Recall that $C = \Sober (f) (D)$ is irreducible closed.

  Since $f$ is initial, $D = f^{-1} (C')$ for some closed subset $C'$
  of $Z$.  Then $f [f^{-1} (C')]$ is included in $C'$, so its closure
  $C \eqdef cl (f [D])$ is also included in $C'$.  In particular,
  $f^{-1} (C)$ is included in $f^{-1} (C') = D$.  Conversely, for
  every $y \in D$, $f (y)$ is in $f [D]$ hence in $C$, so $D$ is
  included in $f^{-1} (C)$.  Therefore $D = f^{-1} (C)$, where $C$ is
  irreducible closed.
\end{proof}
Note that Lemma~\ref{lemma:irred:f-1} does not say that every set
$f^{-1} (C)$, $C \in \Sober Z$, is irreducible closed in $Y$, just
that every irreducible closed subset of $Y$ must be of that form.  We
have a complete characterization when the image of $f$ is open or
closed:
\begin{lemma}
  \label{lemma:irred:f-1:clop}
  Let $f \colon Y \to Z$ be an initial map.
  \begin{enumerate}
  \item If the image of $f$ is open, then the irreducible closed
    subsets of $Y$ are exactly those sets of the form $f^{-1} (C)$
    where $C$ ranges over the irreducible closed subsets of $Z$ that
    intersect the image of $f$.
  \item If the image of $f$ is closed, then the irreducible closed
    subsets of $Y$ are exactly those sets of the form $f^{-1} (C)$
    where $C$ ranges over the irreducible closed subsets of $Z$ that
    are included in the image of $f$.
  \end{enumerate}
\end{lemma}
\begin{proof}
  1. Let us assume that the image $\Img f$ of $f$ is open.  By
  Lemma~\ref{lemma:irred:f-1}, every irreducible closed subset of $Y$
  must be of the form $f^{-1} (C)$, with $C$ irreducible closed in
  $Z$.  Necessarily, $f^{-1} (C)$ must be non-empty, and that implies
  that $C$ intersects the image $\Img f$ of $f$.
  
  In the converse direction, let $C$ be irreducible closed in $Z$, and
  let us assume that $C$ intersects $\Img f$.  Then $f^{-1} (C)$ is
  non-empty.  Let us now consider two open subsets of $Y$ that
  intersect $f^{-1} (C)$.  Since $f$ is initial, they must be of the
  form $f^{-1} (U)$ and $f^{-1} (V)$, where $U$ and $V$ are open in
  $Z$.  Since $f^{-1} (C)$ intersects $f^{-1} (U)$, there is a point
  $y \in Y$ such that $f (y)$ is in $C$ and in $U$, and therefore
  $\Img f \cap U \cap C$ is non-empty.  In other words, $C$ intersects
  $\Img f \cap U$.  Similarly, $C$ also intersects $\Img f \cap V$.
  Both $\Img f \cap U$ and $\Img f \cap V$ are open in $Z$.  Since $C$
  is irreducible, it must intersect their intersection, which is
  $\Img f \cap U \cap V$.  Hence there is a point $f (y)$ (with
  $y \in Y$) in $\Img f$ which is also in $C$, $U$, and $V$.  Then $y$
  is in $f^{-1} (C)$, $f^{-1} (U)$ and $f^{-1} (V)$, so $f^{-1} (C)$
  intersects $f^{-1} (U) \cap f^{-1} (V)$.  We conclude that
  $f^{-1} (C)$ is irreducible.

  2. We now assume that $\Img f$ is closed.  By
  Lemma~\ref{lemma:irred:f-1}, every irreducible closed subset $D$ of
  $Y$ must be of the form $f^{-1} (C)$, where $C \eqdef cl (f [D])$ is
  irreducible closed in $Z$.  Since $f [D] \subseteq \Img f$ and
  $\Img f$ is closed, $cl (f [D])$ is also included in $\Img f$, so
  $C$ is included in the image of $f$.

  In the converse direction, let $C$ be irreducible closed in $Z$, and
  let us assume that $C \subseteq \Img f$.  Since $C$ is non-empty,
  $f^{-1} (C)$ is non-empty.  Let us now consider two closed subsets
  of $Y$ whose union contains $f^{-1} (C)$.  Since $f$ is initial,
  they must be of the form $f^{-1} (C_1)$ and $f^{-1} (C_1)$, where
  $C_1$ and $C_2$ are closed in $Z$.  For every $z \in C$, we can
  write $z$ as $f (y)$ for some $y \in Y$ since $C \subseteq \Img f$.
  Then $y$ is in $f^{-1} (C)$, hence in $f^{-1} (C_1)$ or in
  $f^{-1} (C_2)$.  It follows that $z$ is in $C_1$ or in $C_2$.  This
  shows that $C$ is included in $C_1 \cup C_2$, hence in $C_1$ or in
  $C_2$, using irreducibility.  In the first case, $f^{-1} (C)$ is
  included in $f^{-1} (C_1)$, otherwise in $f^{-1} (C_2)$.  We
  conclude that $f^{-1} (C)$ is irreducible.
\end{proof}


\section{Infinite words}
\label{sec:infinite-words}

Let $X$ be an alphabet, by which we simply mean a topological space,
not necessarily finite.  An \emph{infinite word} on $X$ is an infinite
sequence of elements of $X$, i.e., a function from $\nat$ to $X$.  We
let $X^\omega$ denote the set of all infinite words on $X$.  We write
every $w \in X^\omega$ as $w_0 w_1 \cdots w_n \cdots$, where
$w_n \in X$.  We also write $w_{<n}$ for the length $n$ prefix
$w_0 w_1 \cdots w_{n-1}$ of $w$, and $w_{\geq n}$ for the remainder
$w_n w_{n+1} \cdots$.

We will also consider the set of \emph{finite-or-infinite words}
$X^{\leq \omega} \eqdef X^\omega \cup X^*$.  Those can be defined as
the functions $w$ from an initial segment $\dom w$ of $\nat$ to $X$.

There is a standard quasi-ordering $\leq^\omega$ on $X^{\leq \omega}$,
defined by $w \leq^\omega w'$ if and only if $w$ is a \emph{subword}
of $w'$, namely if there is a monotonic, injective map
$f \colon \dom w \to \dom w'$ such that $w_n \leq w'_{f (n)}$ for
every $n \in \dom w$.  (As usual, $\leq$ is the specialization
preordering of $X$.)

The topology we will be interested in is the following.
We reuse the notation $\langle U_1; U_2; \cdots; U_n \rangle$ to
denote the set of finite \emph{or infinite} words that have a (finite)
subword in $U_1 U_2 \cdots U_n$.  The context should make clear
whether we reason in $X^*$ or in $X^{\leq \omega}$.  The notation
$\langle U_1; U_2; \cdots; U_n; \infy U\rangle$ denotes the set of
(necessarily infinite) words that can be written as a concatenation
$uw$ where $u$ is a finite word in
$\langle U_1; U_2; \cdots; U_n \rangle$ and $w$ contains infinitely
many letters from $U$.
\begin{definition}
  \label{defn:word:limsup}
  The \emph{asymptotic subword topology} on $X^{\leq \omega}$ is
  generated by the subbasic open sets
  $\langle U_1; U_2; \cdots; U_n \rangle$ and
  $\langle U_1; U_2; \cdots; U_n; \infy U \rangle$, where
  $n \in \nat$, and $U_1$, \ldots, $U_n$, $U$ are open in $X$.
\end{definition}
Note that $X^\omega$ is an open subset of $X^{\leq\omega}$, since
$X^\omega = \langle \infy X \rangle$.  It follows that $X^*$ occurs as
a closed subset of $X^{\leq\omega}$.

We equip each subspace of $X^{\leq\omega}$ with the subspace topology.
In particular, we will call \emph{asymptotic subword topology} on
$X^\omega$ the subspace topology.  The subspace topology on $X^*$
happens to coincide with the word topology.

\begin{fact}
  \label{fact:word:limpsup:dc}
  Every open (resp., closed) subset of $X^{\leq \omega}$ in the
  asymptotic subword topology is upwards-closed (resp.,
  downwards-closed) with respect to $\leq^\omega$.  \qed
\end{fact}
We will now show that, if $X$ is Noetherian, then the asymptotic
subword topology is the join of two simpler topologies, the prefix and
the suffix topology.  (The \emph{join} of two topologies is the
coarsest topology that is finer than both.  It has a base of open sets
of the form $U_1 \cap U_2$, where $U_1$ is open in the first topology
and $U_2$ is open in the other one.)

\subsection{The prefix topology}
\label{sec:prefix-topology}

For every $w \in X^{\leq \omega}$, the set
$\pref (w) \eqdef cl (\{w_{<n} \mid n \in \dom w\})$ is irreducible in
$X^*$: if $\mathcal U_1$ and $\mathcal U_2$ are two open subsets of
$X^*$ that intersect $\pref (w)$, then $w_{<m} \in \mathcal U_1$ and
$w_{<n} \in \mathcal U_2$ for some $m, n \in \dom w$; since all open
subsets of $X^*$ are upwards-closed in $\leq^*$, $w_{<\max (m, n)}$ is
both in $\mathcal U_1$ and in $\mathcal U_2$, showing that $\pref (w)$
intersects $\mathcal U_1 \cap \mathcal U_2$.
\begin{definition}
  \label{defn:pref}
  The \emph{prefix map} $\pref \colon X^{\leq\omega} \to \Sober (X^*)$
  is defined by $\pref (w) \eqdef cl (\{w_{<n} \mid n \in \dom w\})$
  for every $w \in X^{\leq \omega}$.

  The \emph{prefix topology} on $X^{\leq\omega}$ is the coarsest that
  makes $\pref$ continuous.
\end{definition}
In other words, a subbase of the prefix topology is given by sets of
the form $\pref^{-1} (\diamond U)$, where $U$ is open in $X^*$.

\begin{remark}
  \label{rem:pref:PF*}
  Let us take $X$ Noetherian.  Since $\pref (w)$ is in $\Sober (X^*)$,
  one must be able to write it as a word product.  The closed subsets
  $F_n \eqdef cl (\{w_m \mid m \in \dom w, m \geq n\})$, $n \in \nat$,
  form a descending sequence, so there is an index $n_0$ such that for
  every $n \geq n_0$, $F_n=F_{n_0}$.  Although we will not use it, one
  can show that
  $\pref (w) = (\dc w_0)^? (\dc w_1)^? \cdots (\dc w_{n-1})^? F_n^*$
  for every $n \geq n_0$.  We leave this as an exercise to the reader.
  As a hint, first, show that the right-hand side
  $\mathcal C \eqdef (\dc w_0)^? (\dc w_1)^? \cdots (\dc w_{n-1})^?
  F_n^*$ does not depend on $n \geq n_0$.  Second, show that
  $w_{<n} \in \mathcal C$ for every $n \in \nat$, and deduce that
  $\pref (w) \subseteq \mathcal C$.  In the reverse direction,
  consider any basic open set
  $\mathcal U \eqdef \langle U_1; U_2; \cdots ; U_k \rangle$ that
  intersects $\mathcal C$, say at
  $u \eqdef u_0 a_1 u_1 a_2 u_2 \cdots u_{k-1} a_k u_k$, where each
  $u_i$ is in $X^*$, each $a_i$ is in $U_i$.  By picking $n$ larger
  than the length of $u$ in the definition of $\mathcal C$, observe
  that $a_1 a_2 \cdots a_k$ is a subword of $w_{<n}$, and conclude
  that $\mathcal U$ intersects $\pref (w)$ at $w_{<n}$.
\end{remark}

For the next lemma, we recall that the sets of the form
$\langle U_1; U_2; \cdots; U_n \rangle$, where each $U_i$ is open in
$X$, form a base, not just a subbase of the asymptotic subword
topology on $X^*$.
\begin{lemma}
  \label{lemma:word:prefix:limsup}
  The prefix map $\pref$ is continuous from $X^{\leq \omega}$ with its
  asymptotic subword topology to $\Sober (X^*)$.  A base of the prefix
  topology is given by the sets
  $\langle U_1; U_2; \cdots; U_n \rangle$, where $U_1$, \ldots, $U_n$
  are open in $X$.

  The prefix topology is coarser than the asymptotic subword topology.
\end{lemma}
\begin{proof}
  Every open subset of $X^*$ is a union of basic open subsets of the
  form $\langle U_1; U_2; \cdots; U_n \rangle$ where each $U_i$ is
  open in $X$.  In order to show that $\pref$ is continuous, since
  $\diamond$ commutes with arbitrary unions, it is enough to show that
  $\pref^{-1} (\diamond (\langle U_1; U_2; \cdots; U_n \rangle))$ is
  open in the asymptotic subword topology.  That is the set of finite
  or infinite words $w$ such that $cl (\{w_{<m} \mid m \in \dom w\})$
  intersects the open set $\langle U_1; U_2; \cdots; U_n \rangle$;
  equivalently, such that some prefix $w_{<m}$ belongs to
  $\langle U_1; U_2; \cdots; U_n \rangle$.  The set
  $\pref^{-1} (\diamond (\langle U_1; U_2; \cdots; U_n \rangle))$ is
  therefore equal to the open subset
  $\langle U_1; U_2; \cdots; U_n \rangle$ of $X^{\leq\omega}$.

  This also shows that the sets
  $\langle U_1; U_2; \cdots; U_n \rangle = \pref^{-1} (\diamond
  (\langle U_1; U_2; \cdots; U_n \rangle))$ form a subbase of the
  prefix topology.

  Since $\pref^{-1}$ and $\diamond$ commute with finite intersections,
  since the sets $\langle U_1; U_2; \cdots; U_n \rangle$ form a base
  of the topology on $X^*$, every finite intersection $U$ of sets of
  the form
  $\pref^{-1} (\diamond (\langle U_1; U_2; \cdots; U_n \rangle))$ can
  be written as $\pref^{-1} (\diamond \mathcal U)$ where $\mathcal U$
  is a union of basic open sets of $X^*$.  Since $\pref^{-1}$ and
  $\diamond$ commute with all unions, $\mathcal U$ is also a union of
  subbasic open sets of the form
  $\langle U_1; U_2; \cdots; U_n \rangle$.  This shows that the given
  subbase is a base.

  The final claim is an immediate consequence of the first one.
\end{proof}

\subsection{The suffix topology}
\label{sec:suffix-topology}

In any complete lattice $L$, the \emph{limit superior} of a sequence
${(u_n)}_{n \in \nat}$ is
$\limsup_{n \in \nat} u_n = \bigwedge_{n \in \nat} \bigvee_{m \geq n}
u_m$.  We will use that notion in lattices of closed sets.  Then
$\limsup_{n \in \nat} C_n = \bigcap_{n \in \nat} cl (\bigcup_{m \geq
  n} C_m)$, where $cl$ is closure.  On a Noetherian space, descending
families of closed sets are stationary, so
$\limsup_{n \in \nat} C_n = cl (\bigcup_{m \geq n} C_m)$ for large
enough $n$.  (The quantifier ``for large enough $n$'' means ``for some
$n_0$, for every $n \geq n_0$''.)

Let the \emph{suffix map} $\suf \colon X^{\leq\omega} \to \Hoare (X)$
be defined by:
\begin{align*}
  \suf (w) & \eqdef \limsup_{n \in \nat} \dc \{w_m \mid m \in \dom w, m \geq n\} \\
  & = \bigcap_{n \in \nat} cl (\{w_m \mid m \in \dom w, m \geq n\}).
\end{align*}
Note that $\suf (w)$ is empty for every finite word $w$.  When $X$ is
Noetherian (or more generally, compact), this is an equivalence: $\suf
(w)=\emptyset$ if and only if $w \in X^*$.

The \emph{suffix topology} on $X^{\leq\omega}$ is the coarsest that
makes $\suf$ continuous.
\begin{lemma}
  \label{lemma:word:suffix:limsup}
  Let $X$ be a Noetherian space.  The suffix map $\suf$ is continuous
  from $X^{\leq\omega}$ with its asymptotic subword topology to
  $\Hoare (X)$.  A subbase of the suffix topology is given by the
  sets $\langle \infy U \rangle$, $U$ open in $X$.

  The suffix topology is coarser than the asymptotic subword topology.
\end{lemma}
\begin{proof}
  A subbase of the suffix topology is given by the sets
  $\suf^{-1} (\Diamond U)$, $U$ open in $X$.  We claim that
  $\suf^{-1} (\Diamond U) = \langle \infy U \rangle$.  This readily
  implies the first and the second claim, and the third one will be an
  immediate consequence.

  Let $w \in X^{\leq\omega}$, and let $n_0$ be such that
  $\limsup_{n \in \nat} \dc \{w_m \mid m \in \dom w, m \geq n\} = cl
  (\{w_m \mid m \in \dom w, m \geq n\})$ for every $n \geq n_0$.  If
  $w \in \suf^{-1} (\Diamond U)$, then for every $n \geq n_0$,
  $cl (\{w_m \mid m \in \dom w, m \geq n\})$ intersects $U$, so
  $w_m \in U$ for some $m \in \dom w$ such that $m \geq n$.  Hence
  $w \in \langle \infy U \rangle$.  Conversely, if
  $w \in \langle \infy U \rangle$, then $\dom w = \nat$ and there are
  infinitely many positions $n \geq n_0$ where $w_n$ is in $U$.  Take
  one.  Then $cl (\{w_m \mid m \in \dom w, m \geq n_0\})$ intersects
  $U$ at $w_n$, showing that $w$ is in $\suf^{-1} (\Diamond U)$.
\end{proof}

\begin{proposition}
  \label{prop:word:limsup:join}
  Let $X$ be a Noetherian space.  The asymptotic subword topology on
  $X$ is the join of the prefix and the suffix topologies.  The
  function
  $\langle \pref, \suf \rangle \colon X^{\leq\omega} \to \Sober (X^*)
  \times \Hoare (X)$ that maps $w$ to $(\pref (w), \suf (w))$ is
  initial.
\end{proposition}
\begin{proof}
  The asymptotic subword topology is finer than both the prefix and
  the suffix topologies, by Lemma~\ref{lemma:word:prefix:limsup} and
  Lemma~\ref{lemma:word:suffix:limsup}.  Conversely, every subbasic
  open set $\langle U_1; U_2; \allowbreak \cdots; U_k \rangle$ is
  prefix-open, and every subbasic open set
  $\langle U_1; U_2; \allowbreak \cdots; U_k ; \infy U\rangle$ is the
  intersection of the prefix-open set
  $\langle U_1; U_2; \allowbreak \cdots; U_k \rangle$ with the
  suffix-open set $\langle \infy U \rangle$.

  By the first part of the lemma, the asymptotic subword topology on
  $X^{\leq \omega}$ is the coarsest that makes both $\pref$ and $\suf$
  continuous, hence the coarsest that makes
  $\langle \pref, \suf \rangle$ continuous.  It follows that
  $\langle \pref, \suf \rangle$ is initial.
\end{proof}

Property (A) follows:
\begin{thm}
  \label{thm:limsup:noeth}
  For every space $X$, $X^{\leq\omega}$ (with the asymptotic subword
  topology) is Noetherian if and only if $X$ is Noetherian.  Similarly
  with $X^\omega$.
\end{thm}
\begin{proof}
  If $X$ is Noetherian, then so are $X^*$, $\Sober (X^*)$,
  $\Hoare (X)$ and their product $\Sober (X^*) \times \Hoare (X)$, as
  we have already seen in Section~\ref{sec:preliminaries} and
  Section~\ref{sec:powersets}.  Since
  $f \eqdef \langle \pref, \suf\rangle$ is initial
  (Proposition~\ref{prop:word:limsup:join}),
  Lemma~\ref{lemma:initial:noeth} ensures that $X^{\leq\omega}$ and
  its subspace $X^\omega$ are Noetherian in the asymptotic subword
  topology.

  In the converse direction, we use the following argument, which
  works in both the $X^{\leq\omega}$ and $X^\omega$ cases.  Let $g$ be
  the function that maps every $x \in X$ to the infinite word
  $x^\omega \eqdef x x \cdots x \cdots$ (in $X^{\leq\omega}$, resp.,
  $X^\omega$).  This is continuous since
  $g^{-1} (\langle U_1; U_2; \allowbreak \cdots; U_k \rangle) = U_1
  \cap U_2 \cap \cdots \cap U_k$ and
  $g^{-1} (\langle U_1; U_2; \allowbreak \cdots; U_k ; \infy U\rangle)
  = U_1 \cap U_2 \cap \cdots \cap U_k \cap U$.  With $k \eqdef 0$, we
  also obtain that every open subset $U$ of $X$ is obtained as
  $g^{-1} (\langle \infy U \rangle)$, so $g$ is initial.  By
  Lemma~\ref{lemma:initial:noeth}, if $X^{\leq\omega}$ (resp.,
  $X^\omega$) is Noetherian, then so is $X$.
\end{proof}
From now on, and unless noted otherwise, we understand
$X^{\leq\omega}$ (and $X^\omega$) with the asymptotic subword
topology.

\subsection{A few useful auxiliary results}
\label{sec:concatenation}

We pause for a moment, and establish two useful results.

\begin{proposition}
  \label{prop:cat:cont}
  The concatenation map
  $cat \colon X^* \times X^{\leq\omega} \to X^{\leq\omega}$ is
  continuous.
\end{proposition}
\begin{proof}
  Let $(u, w)$ be any point in $cat^{-1} (W)$, where
  $W = \langle U_1; U_2 ; \cdots ; U_k \rangle$ (resp.,
  $W = \langle U_1; U_2 ; \cdots ; U_k; \infy U \rangle$), where
  $U_1$, \ldots, $U_k$, and $U$ are open in $X$.  There is an index
  $j$, $0\leq j\leq k$ such that $u$ is in
  $\langle U_1; U_2 ; \cdots ; U_j \rangle$, and $w$ is in
  $\langle U_{j+1}; \cdots ; U_k \rangle$ (resp.,
  $\langle U_{j+1}; \cdots ; U_k ;\infy U \rangle$) Then
  $\langle U_1; U_2 ; \cdots ; U_j \rangle \times \langle U_{j+1};
  \cdots ; U_k \rangle$ (resp.,
  $\langle U_1; U_2 ; \cdots ; U_j \rangle \times \langle U_{j+1};
  \cdots ; U_k ; \infy U \rangle$) is an open neighborhood of $(u, w)$
  that is included in $cat^{-1} (W)$.
\end{proof}

The sets $\langle U_1; U_2 ; \cdots ; U_k \rangle$ and
$\langle U_1; U_2 ; \cdots ; U_k ; \infy U\rangle$ only form a subbase
of $X^\omega$.  We obtain a base as follows.
$\langle U_1; U_2 ; \cdots ; U_k; \infy {V_1} \cap \cdots \cap \cdots
\infy {V_\ell} \rangle$ denotes the set of finite-or-infinite words
that contain letters from $U_1$, $U_2$, \ldots, $U_k$ in that order,
followed by a suffix that contains contains infinitely many letters
from $V_1$, and also infinitely many from $V_2$, \ldots, and
infinitely many from $V_\ell$.  We allow $\ell$ to be equal to $0$; if
$\ell \neq 0$, then that set only contains infinite words.
\begin{lemma}
  \label{lemma:limsup:base}
  Let $X$ be a Noetherian space.  A base of the asymptotic subword
  topology on $X^{\leq\omega}$ is given by the subsets
  $\langle U_1; U_2 ; \cdots ; U_k; \infy {V_1} \cap \cdots \cap
  \cdots \infy {V_\ell} \rangle$ where $U_1$, \ldots, $U_n$, $V_1$,
  \ldots, $V_\ell$ are open in $X$.
\end{lemma}
\begin{proof}
  As a consequence of Proposition~\ref{prop:word:limsup:join}, a base
  is given by intersections of one element
  $\langle U_1; U_2 ; \cdots ; U_k \rangle$ of the base of the prefix
  topology given in Lemma~\ref{lemma:word:prefix:limsup}, and of one
  element of a base of the suffix topology.  For the latter, one can
  take finite intersections
  $\langle \infy {V_1} \rangle \cap \cdots \cap \langle \infy {V_\ell}
  \rangle = \langle \infy {V_1} \cap \cdots \cap \infy {V_\ell}
  \rangle$.
\end{proof}

\subsection{The sobrification of \texorpdfstring{$X^{\leq\omega}$}{X≤ω}}
\label{sec:sobrification-xomega}



We extend the notion of word product to \emph{finite-or-infinite word
  products}: $\omega$-regular expressions of the form
$P F^{\leq\omega}$, where $P$ is a word product and $F$ is a closed
subset of $X$.  $P F^{\leq\omega}$ denotes the sets of finite or
infinite words that are obtained as the concatenation of a finite word
in $P$ with a finite-or-infinite word whose letters are all in $F$.
We also write $P F^\omega$ for $P F^{\leq\omega} \cap X^\omega$: this
is the set of infinite words obtained as the concatenation of a finite
word in $P$ with an infinite word whose letters are all in $F$.  This
is empty if $F$ is empty.  We call \emph{infinite word products} the
expressions $P F^\omega$ where $P$ is a word product and $F$ is a
\emph{non-empty} closed subset of $X$.
\begin{proposition}
  \label{prop:irred=>prod}
  Let $X$ be a Noetherian space.  Every irreducible closed subset of
  $X^{\leq\omega}$ is an finite-or-infinite word product.
\end{proposition}
\begin{proof}
  Let $D$ be an irreducible closed subset of $X^{\leq\omega}$.  The
  map $f \eqdef \langle \pref, \suf \rangle$ is initial by
  Proposition~\ref{prop:word:limsup:join}.  We can therefore apply
  Lemma~\ref{lemma:irred:f-1} to $f$, and we obtain that $D$ is the
  inverse image of some irreducible closed subset of
  $\Sober (X^*) \times \Hoare (X)$ by $f$.  Since the latter space is
  already sober ($\Hoare (X)$ is sober by Lemma~\ref{lemma:H:sobr}, or
  by \cite[Proposition~1.7]{Schalk:PhD}), an irreducible subset is
  just the downward closure of a point $(P, F)$ of
  $\Sober (X^*) \times \Hoare (X)$.  Here $P$ must be a word product
  $A_1 A_2 \cdots A_N$, and $F$ must be a closed subset of $X$.

  Hence $D$ is the set of words $w \in X^{\leq\omega}$ such that
  $\pref (w) \subseteq P$ and $\suf (w) \subseteq F$, that is, such
  that all the finite prefixes of $w$ are in $P$ and the letters $w_n$
  are in $F$ for $n \in \dom w$ large enough.  There may be different
  choices of the pair $(P, F)$, and we choose one such that the number
  $N$ of atoms in $P$ is minimal, and such that given that $N$, $F$ is
  minimal with respect to inclusion.  This is possible since $X$ is
  Noetherian.

  If $F$ is empty, then $D$ is the set of finite words $w$ such that
  $\pref (w) \subseteq P$, or equivalently such that every prefix of
  $w$ is in $P$.  Therefore $D = P$, and that can also be written as
  $P F^{\leq\omega}$, since $F = \emptyset$.  Henceforth let us assume
  that $F$ is non-empty.

  We note that: $(*)$ $D$ is not included in $X^*$.  Indeed,
  otherwise, for every $w \in D$, $w$ would be finite, so $\suf (w)$
  would be empty.  It would follow that $D$ would be included in
  $f^{-1} (\dc (P, \emptyset))$, and that would contradict the
  minimality of $F$.

  Let us write $P$ as a product $A_1 A_2 \cdots A_N$ of atoms, with
  $N$ minimal.  Since $\emptyset^* = \{\epsilon\}$, we may simply
  erase all atoms of the form $A^*$ with $A$ empty: since $N$ is
  minimal, no $A_i$ is of the form $\emptyset^*$.  We claim that
  $N \geq 1$ and that $A_N$ is of the form ${F'}^*$ for some
  (necessarily non-empty) closed set $F'$.  We cannot have $N=0$,
  since that would imply that for every $w \in D$,
  $\pref (w) = \{\epsilon\}$, hence that $D = \{\epsilon\}$; this is
  impossible since $\suf (\{\epsilon\}) = \emptyset$, contradicting
  the fact that $F$ is non-empty.  Hence let us write $P$ as $P' A_N$.
  We now assume that $A_N$ is of the form $C^?$ with $C$ irreducible
  closed, and we aim for a contradiction.  For every infinite word
  $w \in D$, we note that $\pref (w)$ is included in $P'$: for every
  finite prefix $w_{<n}$ of $w$, the finite prefix
  $w_{<n+1} = w_{<n} w_n$ is in $\pref (w)$, hence in $P = P' C^?$,
  and that implies that $w_{<n}$ is in $P'$; since $n$ is arbitrary,
  $\{w_{<n} \mid n \in \dom w\}$ is included in $P'$, and therefore
  $\pref (w) \subseteq P'$, since $P'$ is closed.  Hence every
  infinite word in $D$ is included in $\pref^{-1} (P')$.
  Alternatively, $D$ is included in the union of the set $X^*$ of
  finite words and of $\pref^{-1} (P')$.  Since $D$ is irreducible,
  and both $X^*$ and $\pref^{-1} (P')$ are closed (the latter by
  Lemma~\ref{lemma:word:prefix:limsup}), $D$ must be included in one
  of them.  By $(*)$, $D$ is not included in $X^*$, so $D$ is included
  in $\pref^{-1} (P')$.  It follows that $D$ is the set of words
  $w \in X^{\leq \omega}$ such that $\pref (w) \subseteq P'$ (not just
  $P$) and $\suf (w) \subseteq F$, and this contradicts the minimality
  of $N$.

  We have shown that $P$ is of the form $P' {F'}^*$ for some non-empty
  closed subset $F'$ of $X$.  We now claim that $F$ must be included
  in $F'$.  For every infinite word $w$ in $D$,
  $\pref (w) \subseteq P$, so every finite prefix of $w$ is in $P$.
  Then, either every finite prefix of $w$ is in $P'$---in which case
  $\pref (w) \subseteq P'$---or there is a largest $n_0 \in \nat$ such
  that $w_{<n_0}$ is in $P'$.  In the latter case, every letter $w_n$
  with $n \geq n_0$ must be in $F'$.  For $n_1$ large enough,
  $\suf (w) = cl (\{w_m \mid m \geq n_1\}$.  By picking $n_1$ larger
  than $n_0$, every letter $w_n$ with $n \geq n_1$ is also in $F$.
  Hence $w_n$ is in $F \cap F'$ for every $n \geq n_1$, from which we
  deduce that $\suf (w) \subseteq F \cap F'$.  We have shown that
  every infinite word $w$ in $D$ is in
  $\pref^{-1} (\dc P') \cup \suf^{-1} (\dc (F \cap F'))$, or
  equivalently, that $D$ is included in the union of $X^*$,
  $\pref^{-1} (\dc P')$, and $\suf^{-1} (\dc (F \cap F'))$.  Those
  three sets are closed, using Lemma~\ref{lemma:word:prefix:limsup} in
  the case of the second one, and Lemma~\ref{lemma:word:suffix:limsup}
  for the third one.  Since $D$ is irreducible, it must be included in
  one of them.  It is not included in $X^*$ by $(*)$.  It is not
  included in $\pref^{-1} (\dc P')$, otherwise $D$ would be included
  in $f^{-1} (\dc (P', F))$, contradicting the minimality of $N$.
  Therefore $D$ is included in $\suf^{-1} (\dc (F \cap F'))$.  This
  entails that $D$ is included in $f^{-1} (\dc (P, F \cap F'))$.
  Since $F$ is minimal, $F = F \cap F'$, hence $F$ is included in
  $F'$.

  Now that we know that $D = f^{-1} (\dc (P' {F'}^*, F))$ with
  $F \subseteq F'$, we verify that $D = P' {F'}^* F^{\leq\omega}$.
  For every $w \in D$, $\suf (w) \subseteq F$ so there is an index
  $n_0$ such that, for every $n \in \dom w$ with $n \geq n_0$, $w_n$
  is in $F$.  If $w$ is a finite word, we may take $n_0$ equal to one
  plus the length of $w$.  Whatever the case, $w_{\geq n}$ is in
  $F^{\leq\omega}$.  Since $\pref (w) \subseteq P' {F'}^*$, $w_{<n}$
  is in $P' {F'}^*$, so $w$ is in $P' {F'}^* F^{\leq\omega}$.
  Conversely, let $w \in P' {F'}^* F^{\leq\omega}$.  If $w$ is finite,
  then it is in $P' {F'}^* F^* \subseteq P' {F'}^*$ (since
  $F \subseteq F'$), so $\pref (w) \subseteq P' {F'}^*$; also,
  $\suf (w) = \emptyset \subseteq F$, so $w$ is in
  $f^{-1} (\dc (P' {F'}^*, F)) = D$.  Hence let us assume that $w$ is
  infinite.  There is an $n_0 \in \nat$ such that
  $w_{<n_0} \in P' {F'}^*$ and $w_n \in F$ for every $n \geq n_0$.  In
  particular, $\suf (w)$ is included in
  $cl (\{w_n \mid n \in \dom w, n \geq n_0\})$, hence in $F$.  For
  every finite prefix $w_{<n}$ of $w$, either $n \leq n_0$, in which
  case $w_{<n}$ is a subword of $w_{<n_0}$ hence is in $P' {F'}^*$, or
  $n > n_0$, in which case we can write $w_{<n}$ as $w_{<n_0}w'$ where
  $w' \in {F}^*$.  Since $F \subseteq F'$, $w_{<n}$ is then again in
  $P' {F'}^*$.  It follows that $\pref (w) \subseteq P' {F'}^*$, hence
  that $f (w) \in \dc (P' {F'}^*, F)$, so $w$ is in $D$.
\end{proof}
The last part of the previous proof shows the useful fact that for
every word product $P$, for all closed subsets $F$ and $F'$ of $X$
such that $F \subseteq F'$,
$P {F'}^* F^{\leq\omega} = \langle \pref, \suf\rangle^{-1} (\dc (P
{F'}^*, F))$.  When $F=F'$, and since
$PF^*F^{\leq\omega} = PF^{\leq\omega}$ (resp.,
$PF^\omega = PF^{\leq\omega} \cap X^\omega$), we obtain:
\begin{fact}
  \label{fact:product}
  Let $X$ be a Noetherian space.  For every word product $P$, for
  every closed subset $F$ of $X$,
  $P F^{\leq\omega} = \langle \pref, \suf\rangle^{-1} (\dc (P F^*,
  F))$.  In particular, every finite-or-infinite word product
  $P F^{\leq\omega}$ is closed in $X^{\leq\omega}$ (resp., every
  infinite word product $PF^\omega$ is closed in $X^\omega$).  \qed
\end{fact}
Recall that we have required the closed set $F$ to be non-empty in
infinite products $PF^\omega$.  This is unimportant in
Fact~\ref{fact:product}, since $P \emptyset^\omega = \emptyset$ is
closed anyway, but it matters in the following.

\begin{thm}
  \label{thm:prod=irred}
  Let $X$ be a Noetherian space.  The irreducible closed subsets of
  $X^{\leq\omega}$ (resp., $X^\omega$) are the finite-or-infinite word
  products (resp., the infinite word products).
\end{thm}
\begin{proof}
  We deal with $X^{\leq\omega}$ first.  Considering
  Proposition~\ref{prop:irred=>prod} and Fact~\ref{fact:product}, it
  remains to show that every finite-or-infinite word product
  $P F^{\leq\omega}$ is irreducible (where $P$ is a word product, and
  $F$ is closed in $X$).

  We start by showing that $F^{\leq\omega}$ is irreducible in
  $X^{\leq\omega}$.  It is more, namely it is \emph{directed}, with
  respect to the quasi-ordering $\leq^\omega$.  In other words, we
  show that $F^{\leq\omega}$ is is non-empty (which is clear, since it
  contains the empty word $\epsilon$), and that any two elements $w_1$
  and $w_2$ of $F^{\leq\omega}$ have an upper bound in
  $F^{\leq\omega}$ with respect to $\leq^\omega$.  For that upper
  bound, we can simply take: the concatenation $w_1 w_2$ if $w_1$ is
  finite (or $w_2 w_1$ if $w_2$ is finite), and the one-for-one
  interleaving of $w_1$ and $w_2$ if both are infinite (i.e., the
  letters at even positions are those from $w_1$, the letters at odd
  positions are those from $w_2$).  Every set that is directed with
  respect to $\leq^\omega$ is irreducible: if it intersects two open
  sets $U_1$ (say at $w_1$) and $U_2$ (say at $w_2$) then it
  intersects $U_1 \cap U_2$ (at the chosen upper bound of $w_1$ and
  $w_2$ in the directed set), using Fact~\ref{fact:word:limpsup:dc}.

  Given any word product $P$, we know that $P$ is irreducible closed
  in $X^*$, so $P \times F^{\leq\omega}$ is irreducible closed in
  $X^* \times X^{\leq\omega}$.  Then, using the fact that $cat$ is
  continuous (Proposition~\ref{prop:cat:cont}),
  $cl (cat [P \times F^{\leq\omega}])$ is irreducible closed.  (Recall
  from Section~\ref{sec:preliminaries} that
  $cl (f [C]) = \Sober (f) (C)$ is irreducible closed for every
  irreducible closed subset $C$ and every continuous map $f$.)
  Evidently, $cat [P \times F^{\leq\omega}] = P F^{\leq\omega}$, and
  $cl (P F^{\leq\omega}) = P F^{\leq\omega}$ by
  Fact~\ref{fact:product}, so $P F^{\leq\omega}$ is irreducible closed
  in $X^{\leq\omega}$.

  In the case of $X^\omega$, the inclusion map $i$ of $X^\omega$ into
  $X^{\leq\omega}$ is a topological embedding (hence initial), and its
  image $X^\omega = \langle \infy X \rangle$ is open.  By
  Lemma~\ref{lemma:irred:f-1:clop}, item~1, the irreducible closed
  subsets of $X^\omega$ are the sets of the form
  $i^{-1} (P F^{\leq\omega}) = P F^{\leq\omega} \cap X^\omega$ ($P$
  word product, $F$ closed) such that $P F^{\leq\omega}$ intersects
  $X^\omega$.  Those are the sets of the form $P F^\omega$ where,
  additionally, $F$ is non-empty.
\end{proof}

\subsection{The specialization preordering on \texorpdfstring{$X^{\leq\omega}$}{X≤ω} and Property (B)}
\label{sec:spec-preord-xomega}

Our first step in characterizing the specialization preordering on
$X^{\leq \omega}$ is to characterize the downward closure,
equivalently, the closure, of its points.
\begin{lemma}
  \label{lemma:dc}
  Let $X$ be a Noetherian space.  Given a fixed word
  $w \in X^{\leq\omega}$, let $n_0$ be such that
  $\suf (w) = cl (\{w_m \mid m \in \dom w, m \geq n\})$ for every
  $n \geq n_0$.  For every $n \geq n_0$:
  \begin{enumerate}
  \item the closure of $w$ in $X^{\leq\omega}$ is
    $(\dc w_0)^? (\dc w_1)^? \cdots (\dc w_{n-1})^? (\suf
    (w))^{\leq\omega}$;
  \item if $w \in X^\omega$, then its closure in $X^\omega$ is
    $(\dc w_0)^? (\dc w_1)^? \cdots (\dc w_{n-1})^? (\suf
    (w))^\omega$.
  \end{enumerate}
\end{lemma}
\begin{proof}
  Let $P \eqdef (\dc w_0)^? (\dc w_1)^? \cdots (\dc w_{n-1})^?$,
  $F \eqdef \suf (w)$, and $C \eqdef P F^{\leq\omega}$.
  
  1.  $C$ certainly contains $w$, and is closed by
  Fact~\ref{fact:product}, so $cl (\{w\}) \subseteq C$.  (We write
  $cl$ for closure in $X^{\leq\omega}$ here.)
  
  We turn to the converse implication.  Since $cl (\{w\})$ is
  irreducible closed (the closures of points are always irreducible
  closed), it must be a finite-or-infinite word product
  $P' {F'}^{\leq\omega}$ by Theorem~\ref{thm:prod=irred}.
  $P' {F'}^{\leq\omega}$ is equal to
  $\langle \pref, \suf \rangle^{-1} (\dc (P' {F'}^*, F'))$ by
  Fact~\ref{fact:product}.  Since $w$ is in its closure
  $P' {F'}^{\leq\omega}$, $\pref (w)$ is included in $P' {F'}^*$ and
  $\suf (w)$ is included in $F'$.  The latter means that
  $F \subseteq F'$.  Using the former, we claim that
  $P F^* \subseteq P' {F'}^*$.  Since $F \subseteq F'$, it is
  equivalent to show that $P \subseteq P' {F'}^*$.  That is obvious,
  since every element of $P$ is a subword of $w_{<n}$, hence a
  subword of a finite prefix of $w$, hence belongs to $\pref (w)$,
  which is included in $P' {F'}^*$.

  Now that $P F^* \subseteq P' {F'}^*$ and $F \subseteq F'$,
  $C = \langle \pref, \suf \rangle^{-1} (\dc (P F^*, F))$
  (Fact~\ref{fact:product}) is included in
  $\langle \pref, \suf \rangle^{-1} (\dc (P' {F'}^*, F')) = cl
  (\{w\})$.

  2. The closure of $w$ in $X^\omega$ is
  $cl (\{w\}) \cap X^\omega = P F^{\leq\omega} \cap X^\omega = P
  F^\omega$.
\end{proof}

Lemma~\ref{lemma:dc} yields a description of the specialization
preordering of $X^{\leq\omega}$ and of $X^\omega$, since $w'$ is below
$w$ in that ordering if and only if $w'$ is in the closure of $w$.
That is far from explicit.

We can improve on that situation when $X$ is a wqo, obtaining an
analogue of Property (B) for $X^{\leq\omega}$ and $X^\omega$.

\begin{lemma}
  \label{lemma:suf:c}
  If $X$ is a wqo, then for every $w \in X^{\leq\omega}$, $\suf (w)$
  is the set of letters that are below infinitely many letters from
  $w$, and is equal to $\bigcup_{m \in \dom w, m \geq n} \dc w_m$ for
  $n$ large enough.
\end{lemma}
\begin{proof}
  If $w$ is finite, then $\suf (w)$ is empty, and the claim is clear.
  Let us assume that $w$ is an infinite word.  Since $X$ is a wqo, for
  every $n \in \nat$, $cl (\{w_m \mid m \geq n\})$ is equal to
  $\dc \{w_m \mid m \geq n\} = \bigcup_{m \geq n} \dc w_m$.  Hence
  $\suf (w) = \bigcap_{n \in \nat} \bigcup_{m \geq n} \dc w_m$ is the
  set of letters that are below infinitely many letters from $w$.
  Since $\suf (w) = \dc \{w_m \mid m \geq n\}$ for $n$ large enough,
  it is also equal to $\bigcup_{m \geq n} \dc w_m$ for $n$ large
  enough.
\end{proof}

\begin{proposition}
  \label{prop:spec}
  If $X$ is a wqo, then the specialization preordering on
  $X^{\leq\omega}$ is the subword preordering $\leq^\omega$.
\end{proposition}
\begin{proof}
  Let us fix $w \in X^{\leq\omega}$.  It suffices to show that the
  closure of $w$ is exactly the set of finite-or-infinite subwords of
  $w$.  By Fact~\ref{fact:word:limpsup:dc}, every subword of $w$ is in
  the closure of $w$.  Conversely, let $w'$ be any element of the
  closure of $w$.  Using Lemma~\ref{lemma:dc} and
  Lemma~\ref{lemma:suf:c}, there is a natural number $n_0$ such that,
  for every $n \geq n_0$,
  $\suf (w) = \bigcup_{m \in \dom w, m \geq n} \dc w_m$, and then $w'$
  is in
  $(\dc w_0)^? (\dc w_1)^? \allowbreak \cdots (\dc w_{n-1})^? (\suf
  (w))^{\leq\omega}$.

  We first use this formula for $\suf (w)$ with $n \eqdef n_0$.  Then
  $w' = us$ where $u$ is in
  $(\dc w_0)^?  (\dc w_1)^? \cdots (\dc w_{n_0-1})^?$, hence is a
  subword of $w_{<n_0}$, and $s$ is in $(\suf (w))^{\leq\omega}$.  The
  latter means that $s$ is an (infinite) word whose letters are all in
  $\suf (w)$.  We claim that $s$ is a subword of $w_{\geq n_0}$,
  namely that there is a monotonic injective map
  $f \colon \dom s \to \dom w_{\geq n_0}$ such that
  $s_i \leq w_{\geq n_0} (f (i)) = w_{f (i)+n_0}$ for every
  $i \in \dom s$.  We define $f (i)$ by induction on $i \in \dom s$ as
  follows.  If $i=0$ is in $\dom s$, then $s_0$ is in $\suf (w)$.
  Using the formula for $\suf (w)$ with $n \eqdef n_0$, $s_0$ is in
  $\dc w_m$ for some $m \geq n_0$.  We define $f (0)$ as $m-n_0$.  For
  every non-zero $i \in \dom s$, and remembering that $f (i-1)$ is
  already defined by induction hypothesis, we use the formula for
  $\suf (w)$ with $n \eqdef f (i-1)+n_0+1$.  Then $s_{i+1}$ is in
  $\dc w_m$ for some $m > f (i-1)+n_0$, and we let $f (i)$ be $m-n_0$,
  so that $f (i) > f (i-1)$ and $s_i \leq w_{f (i)+n_0}$.

  This establishes that $s$ is a subword of $w_{\geq n_0}$.  It
  follows that $w' = us$ is a subword of $w_{<n_0} w_{\geq n_0} = w$.
\end{proof}

\subsection{Property (C)}
\label{sec:property-c}

We now investigate when $X^{\leq\omega}$ and $X^\omega$ are themselves
wqos.  In particular, this means when their topology is Alexandroff.
As with powersets, this is a different question from asking when
$\leq^\omega$ is a well-quasi-ordering on $X^\omega$, which is
equivalent to $\leq$ being an $\omega^2$-wqo.  (That equivalence is
the special case $\alpha=\omega^2$ of Theorem~2.8 of
\cite{Marcone:bqo}, paying attention that what Marcone calls
$\alpha$-wqo is what we call $\omega^\alpha$-wqo---some authors also
use the term $\omega^\alpha$-bqo.)

\begin{proposition}
  \label{prop:words:finite}
  If $X$ is essentially finite, then the asymptotic subword topology
  on $X^{\leq\omega}$ (resp., $X^\omega$) is the Alexandroff topology
  of $\leq^\omega$.
\end{proposition}
\begin{proof}
  We only deal with $X^{\leq\omega}$.  The case of $X^\omega$ will
  follow, because the subspace topology of a space with the Alexandroff
  topology of a preordering $\preceq$ is the Alexandroff topology of
  the restriction of $\preceq$ to the subspace.
  
  Considering Proposition~\ref{prop:spec}, it suffices to show that
  every upwards-closed subset of $X^{\leq \omega}$, with respect to
  $\leq^\omega$, is open in the asymptotic subword topology.  To that
  end, it suffices to show that the upward closure $\upc^\omega w$ of
  any $w \in X^{\leq\omega}$ with respect to $\leq^\omega$ is open,
  since every upwards-closed set is a union of such upward closures.

  If $w$ is finite, then $\upc^\omega w = \langle \upc w_0; \upc w_1;
  \cdots; \upc w_{n-1} \rangle$ where $n$ is the length of $w$.  Note
  that each set $\upc w_i$ is open in $X$, because the topology of an
  essentially finite space is always Alexandroff.

  Let us assume $w$ infinite.  Since $X$ is essentially finite, there
  are only finitely many distinct sets of the form $\upc w_n$,
  $n \in \nat$.  Some of them occur at only finitely many positions
  $n$ in $w$: let $n_0$ be any index exceeding all those positions.
  Then every $\upc w_n$, $n \geq n_0$, is also equal to $\upc w_m$ for
  infinitely many indices $m \geq n_0$.  Let $\{V_1, \cdots, V_\ell\}$
  be the (finite, non-empty) set $\{\upc w_n \mid n \geq n_0\}$, and
  let
  $U \eqdef \langle \upc w_0; \upc w_1; \cdots; \upc w_{n_0-1}; \infy
  {V_1} \cap \cdots \cap \infy {V_\ell}\rangle$.  This is open in the
  asymptotic subword topology.  $U$ contains $w$, by construction.
  Using Fact~\ref{fact:word:limpsup:dc}, $\upc^\omega w$ is entirely
  included in $U$.  Conversely, let $w'$ be any element of $U$.  Then
  $w' = us$ where $u$ is a finite word that contains a letter above
  $w_0$, a later letter above $w_1$, \ldots, and a letter above
  $w_{n_0-1}$, and $s \in X^\omega$ contains infinitely many letters
  from each of $V_1$, \ldots, $V_\ell$---in other words, for every
  $n \geq n_0$, $s$ contains infinitely many letters above $w_n$.
  Hence $s$ contains a letter above $w_{n_0}$, then a later letter
  above $w_{n_0+1}$, etc., so $w_{\geq n_0} \leq^\omega s$.  It
  follows that $w \leq^\omega w'$.  Therefore
  $U \subseteq \upc^\omega w$, whence equality follows.
\end{proof}

\begin{proposition}
  \label{prop:words:wqo}
  Let $X$ be a Noetherian space.  The asymptotic subword topology on
  $X^{\leq\omega}$ (resp., $X^\omega$) is Alexandroff if and only if
  $X$ is essentially finite.
\end{proposition}
\begin{proof}
  One direction is by Proposition~\ref{prop:words:finite}.  In the
  converse direction, we assume that $X^\omega$ is Alexandroff, and we
  wish to show that $X$ is essentially finite.  The case of
  $X^{\leq\omega}$ reduces to that case: if $X^{\leq\omega}$ is
  Alexandroff, so is it subspace $X^\omega$.

  Let
  $C_0 \subsetneq C_1 \subsetneq \cdots \subsetneq C_n \subsetneq
  \cdots$ be a strictly ascending sequence of closed subsets of $X$,
  and let $x_n$ be a point of $C_{n+1} \diff C_n$ for every
  $n \in \nat$.  For each $n \in \nat$, $C_n^\omega$ is closed in
  $X^\omega$, by
  Theorem~\ref{thm:prod=irred}.
  Let
  $\mathcal C \eqdef \bigcup_{n \in \nat} C_n^\omega$: since the
  topology of $X^\omega$ is Alexandroff, this is again closed.

  Note that $w \eqdef x_0 x_1 \cdots x_n \cdots$ is in no
  $C_n^\omega$, hence not in $\mathcal C$.  By
  Lemma~\ref{lemma:limsup:base}, there is a basic open subset
  $W \eqdef \langle U_1; U_2; \cdots; U_k; \infy {V_1} \cap \cdots
  \cap \infy {V_\ell}\rangle$ of $X^\omega$ that contains $w$ and is
  disjoint from $\mathcal C$.  Since it contains $w$, we can write $w$
  as $w_{<n_0} w_{\geq n_0}$ where
  $w_{<n_0} \in \langle U_1; U_2; \cdots; U_k \rangle$ and
  $w_{\geq n_0}$ contains infinitely many letters from each of $V_1$,
  \ldots, $V_\ell$.  Let us pick one letter $x_{n_1}$ from
  $w_{\geq n_0}$ in $V_1$, \ldots, one letter $x_{n_\ell}$ from
  $w_{\geq n_0}$ in $V_\ell$.  Then the infinite word
  $w_{<n_0} (x_{n_1} \cdots x_{n_\ell})^\omega$ is in $W$, but it is
  also in $C_{n+1}^\omega$, where $n$ is any natural number exceeding
  $\max (n_0, n_1, \cdots, n_\ell)$.  In particular, $W$ intersects
  $\mathcal C$, which is impossible.

  We conclude that there cannot be any infinite strictly ascending
  sequence of closed subsets of $X$.  By Lemma~\ref{lemma:wqo:acc},
  $X$ must be essentially finite.
\end{proof}

\subsection{An S-representation on \texorpdfstring{$X^{\leq\omega}$}{X≤ω} and on \texorpdfstring{$X^\omega$}{Xω}, and Property (D)}
\label{sec:an-s-representation}

Testing inclusion of finite-or-infinite word products is as easy as
testing inclusion of finite word products.
\begin{lemma}
  \label{lemma:irred:subset}
  Let $X$ be a Noetherian space.
  \begin{enumerate}
  \item For all finite-or-infinite word products $PF^{\leq\omega}$ and
    $P'{F'}^{\leq\omega}$,
    $PF^{\leq\omega} \subseteq P' {F'}^{\leq\omega}$ if and only if
    $PF^* \subseteq P' {F'}^*$ and $F \subseteq F'$.
  \item For all infinite word products $PF^\omega$ and
    $P'{F'}^\omega$, $PF^\omega \subseteq P' {F'}^\omega$ if and only
    if $PF^* \subseteq P' {F'}^*$ and $F \subseteq F'$.
  \end{enumerate}
\end{lemma}
Recall that $F$ and $F'$ are required to be non-empty in infinite word
products, not in finite-or-infinite word products.

\begin{proof}
  We first show: $(i)$ If $F$ is non-empty, then
  $PF^\omega \subseteq P'{F'}^\omega$ implies
  $PF^* \subseteq P' {F'}^*$ and $F \subseteq F'$.  Henceforth, let us
  assume $PF^\omega \subseteq P'{F'}^\omega$.  We first show that $F$
  is included in $F'$.  Since $F$ is non-empty, we can pick an element
  $x$ from $F$.  Note that the empty word $\epsilon$ is in $P$.  Hence
  $x^\omega$ ($=\epsilon x^\omega$) is in $PF^\omega$, and therefore
  in $P' {F'}^\omega$.  This means that we can write $x^\omega$ as
  $uv$ where $u \in P'$ and $v \in {F'}^\omega$; the latter,
  together with the fact that $v = x^\omega$, implies $x \in F'$.  It
  follows that $F \subseteq F'$.  We now claim that
  $PF^* \subseteq P' {F'}^*$.  Let us pick any finite word $w$ from
  $PF^*$.  The infinite word $w x^\omega$ is in $PF^\omega$, hence in
  $P' {F'}^\omega$.  It follows that $w x^\omega$ is of the form $u v$
  where $u \in P'$ and $v \in {F'}^\omega$.  If $u$ is a prefix of
  $w$, then $w$ is equal to the concatenation of $u$ with a prefix of
  $v$, hence is in $P' {F'}^*$.  Otherwise, $w$ is a prefix of $u$.
  Since $u$ is in $P'$, it is also in $P' {F'}^*$, and since the
  latter is closed under $\leq^*$, $w$ is also in $P' {F'}^*$.

  We deduce: $(ii)$ $PF^{\leq\omega} \subseteq P' {F'}^{\leq\omega}$
  implies $PF^* \subseteq P' {F'}^*$ and $F \subseteq F'$.  If $F$ is
  non-empty, then since
  $PF^\omega = PF^{\leq\omega} \cap X^\omega \subseteq
  P'{F'}^{\leq\omega} \cap X^\omega = P'{F'}^\omega$, we can use $(i)$
  and conclude.  Otherwise, since $F=\emptyset$,
  $PF^{\leq\omega} = P$, so $P$ is included in $P'{F'}^{\leq\omega}$.
  Since all the elements of $P$ are finite words, $P$ is in fact
  included in $P'{F'}^{\leq\omega} \cap X^* = P'{F'}^*$.  The
  inequality $F \subseteq F'$ is trivial.
  
  In the converse direction, we have: $(iii)$ if
  $PF^* \subseteq P' {F'}^*$ and $F \subseteq F'$, then
  $PF^{\leq\omega} \subseteq P' {F'}^{\leq\omega}$.  Indeed:
  \begin{align*}
    P F^{\leq\omega}
    & = \langle \pref, \suf\rangle^{-1} (\dc (P F^*, F))
    & \text{(Fact~\ref{fact:product})} \\
    & \subseteq \langle \pref, \suf\rangle^{-1} (\dc (P' {F'}^*, F'))
    = P' {F'}^{\leq\omega}.
  \end{align*}
  Finally: $(iv)$ if $PF^* \subseteq P' {F'}^*$ and $F \subseteq F'$,
  then $PF^\omega \subseteq P' {F'}^\omega$.  This follows from
  $(iii)$ by taking intersections with $X^\omega$.

  Item~1 follows from $(ii)$ and $(iii)$.  Item~2 follows from $(i)$
  and $(iv)$.
\end{proof}


Computing intersections of infinite word products also reduces to the
case of finite word products, as we will see in
Lemma~\ref{lemma:irred:inter} below.  We notice the following.  The
point if that $F_i$ is the same for every $i$.
\begin{lemma}
  \label{lemma:PF*:inter}
  Let $X$ be a Noetherian space.  For all finite word products $P$,
  $P'$, for all closed subsets $F$, $F'$ of $X$, $PF^* \cap P'{F'}^*$
  is a finite union of word products $P_i F_i^*$ where $F_i = F \cap
  F'$ for every $i$.
\end{lemma}
\begin{proof}
  Since $X^*$ is Noetherian, $PF^* \cap P'{F'}^*$ is a finite union
  $\bigcup_{i=1}^n P_i$ where each $P_i$ is a word product.  For every
  $i$, $1\leq i\leq n$, for every $w \in P_i$, for every
  $w' \in (F \cap F')^*$, $ww'$ is in
  $PF^* (F \cap F')^* \subseteq PF^*$ and in
  $P'{F'}^* (F \cap F')^* \subseteq P'{F'}^*$, hence in
  $PF^* \cap P'{F'}^*$, and therefore in some $P_j$, $1\leq j\leq n$.
  It follows that $\bigcup_{i=1}^n P_i (F \cap F')^* \subseteq
  \bigcup_{i=1}^n P_i$.  The converse inclusion is obvious.
\end{proof}

\begin{lemma}
  \label{lemma:irred:inter}
  Let $X$ be a Noetherian space.  Given any two finite-or-infinite
  word products $PF^{\leq\omega}$ and $P' {F'}^{\leq\omega}$, one can
  write $PF^* \cap P' {F'}^*$ as a finite union of finite word
  products of the form $P_i (F \cap F')^*$, $1\leq i\leq n$, and then
  $PF^{\leq\omega} \cap P' {F'}^{\leq\omega} = \bigcup_{i=1}^n P_i (F
  \cap F')^{\leq\omega}$.
\end{lemma}
\begin{proof}
  The fact that $PF^* \cap P' {F'}^*$ can be written as a finite union
  of finite word products of the form $P_i (F \cap F')^*$,
  $1\leq i\leq n$, is by Lemma~\ref{lemma:PF*:inter}.

  Let
  $\mathcal F \eqdef \{C \in \Sober (X^*) \mid C \subseteq PF^* \cap
  P'{F'}^*\}$.  We claim that
  $\mathcal F = \bigcup_{i=1}^n \dc_{\Sober (X^*)} (P_i (F \cap
  F')^*)$.  For every $C \in \mathcal F$, $C$ is included in
  $PF^* \cap P' {F'}^* = \bigcup_{i=1}^n P_i (F \cap F')^*$, hence in
  some $P_i (F \cap F')^*$ because $C$ is irreducible.  Hence $C$ is
  also in $\bigcup_{i=1}^n \dc_{\Sober (X^*)} (P_i (F \cap F')^*)$.
  Conversely, every element $C$ of
  $\bigcup_{i=1}^n \dc_{\Sober (X^*)} (P_i (F \cap F')^*)$ is included
  in some $P_i (F \cap F')^*$, hence in $PF^* \cap P' {F'}^*$.

  Then
  $\dc (P F^*, F) \cap \dc (P' {F'}^*, F') = \mathcal F \times
  \dc_{\Hoare (X)} (F \cap F')$, so:
  \begin{align*}
    PF^{\leq\omega} \cap P'{F'}^{\leq\omega}
    & = \langle \pref, \suf\rangle^{-1} (\dc (P F^*, F))
      \cap \langle \pref, \suf\rangle^{-1} (\dc (P' {F'}^*, F'))
    \\
    & \qquad\qquad\qquad\text{by Fact~\ref{fact:product}} \\
    & = \langle \pref, \suf\rangle^{-1} (\mathcal F \times \dc_{\Hoare (X)} (F
      \cap F')))
    \\
    & = \langle \pref, \suf\rangle^{-1} (\bigcup_{i=1}^n \dc (P_i (F
      \cap F')^*, F \cap F')) \\
    & = \bigcup_{i=1}^n P_i (F \cap F')^{\leq\omega},
  \end{align*}
  by Fact~\ref{fact:product} again.
\end{proof}

We have a similar result for infinite words.  We only have to pay
attention that $P_i (F \cap F')^*$ is an infinite word product (an
irreducible closed set) if and only if $F \cap F'$ is non-empty.
\begin{lemma}
  \label{lemma:irred:inter:infinite}
  Let $X$ be a Noetherian space.  Given any two infinite word products
  $PF^\omega$ and $P' {F'}^\omega$,
  \begin{itemize}
  \item either $F \cap F'$ is empty and
    $PF^\omega \cap P' {F'}^\omega = \emptyset$;
  \item or $F \cap F'$ is non-empty, and one can write
    $PF^* \cap P' {F'}^*$ as a finite union of finite word products of
    the form $P_i (F \cap F')^*$, $1\leq i\leq n$; in that case,
    $PF^\omega \cap P' {F'}^\omega = \bigcup_{i=1}^n P_i (F \cap
    F')^\omega$.
  \end{itemize}
\end{lemma}
\begin{proof}
  If $F \cap F'$ is empty, then it is clear that
  $PF^\omega \cap P' {F'}^\omega = \emptyset$.  Otherwise,
  \begin{align*}
    PF^\omega \cap P' {F'}^\omega
    & = PF^{\leq\omega} \cap P' {F'}^{\leq\omega} \cap X^\omega \\
    & = \bigcup_{i=1}^n P_i (F \cap
      F')^{\leq\omega} \cap X^\omega
    & \text{by Lemma~\ref{lemma:irred:inter}} \\
    & = \bigcup_{i=1}^n P_i (F \cap F')^\omega.
  \end{align*}
\end{proof}

Let us turn to actual S-representations.  We assume an
S-representation $(S, \SEval \_, \unlhd, \allowbreak \topcap, \meet)$
of $X$.  For every finite subset $u \eqdef \{a_1, \cdots, a_n\}$ of
$S$, let $\Eval u$ denote $\Eval {a_1} \cup \cdots \cup \Eval {a_n}$.
For all finite subsets $u$ and $v$ of $S$, we also write $u\meet v$
for $\bigcup_{a \in u, b \in v} a \meet b$, so that
$\Eval {u\meet v} = (\bigcup_{a \in u} \Eval a) \cap (\bigcup_{b \in
  v} \Eval {b}) = \Eval u \cap \Eval v$.

Lemma~\ref{lemma:PF*:inter} has the following computable equivalent,
which says that for \emph{syntactic} word products $Pu^*$ and
$Q {v}^*$, $Pu^* \meet' Q{v}^*$ computes the intersection
$\SEval {Pu^*}' \cap \SEval {Q{v}^*}' = \SEval {P}' \Eval u^* \cap
\SEval {Q}' \Eval {v}^*$ as a finite set of syntactic word products of
the form $P_i {(u\meet v)}^*$.  For this result, we need to use the
optimized version of $\meet'$ described in Remark~\ref{rem:meet'}.
\begin{lemma}
  \label{lemma:PF*:inter:Srepr}
  Let $X$ be a Noetherian space, and
  $(S, \SEval \_, \unlhd, \allowbreak \topcap, \meet)$ be an
  S-representation of $X$.  For all (syntactic) word products of the
  form $Pu^*$ and $Q{v}^*$, their intersection, as computed using
  Proposition~\ref{prop:srepr:X*}, item~(5), and removing subsumed
  word products as per Remark~\ref{rem:meet'}, is a finite set of word
  products of the form $R{(u\meet v)}^*$.
\end{lemma}
\begin{proof}
  By induction on the sum of the length $n$ of $P$ and the length $n'$
  of $Q$.  This is a direct appeal to the induction hypothesis if $n
  \geq 1$ and $n' \geq 1$.  The interesting case is when $n'=0$ (or,
  symmetrically, $n=0$).  If $n=n'=0$, then $Pu^* \meet' Q{v}^* = u^*
  \meet' {v}^* = \{{(u\meet v)}^* R \mid R \in (\varepsilon \meet' {v}^*)
  \cup (u^* \meet' \varepsilon)\} = \{{(u\meet v)}^*\}$, by
  (\ref{eq:inter:**}) and (\ref{eq:inter:eps}).  If $n \geq 1$ and $n'
  = 0$, then we need to show the claim for intersections of the form: (1)
  $a^? P u^* \meet' {v}^*$ and (2) $u_0^* P u^* \meet' {v}^*$.
  
  In case (1), we use (\ref{eq:inter:?*}): the elements of
  $a^? P u^* \meet' {v}^*$ are of the form ${c}^? R$ (where $c$ ranges
  over $u\meet v$, if $u\meet v \neq \emptyset$) or just $R$ (if
  $u\meet v = \emptyset$), where $R \in P u^* \meet' {v}^*$, plus
  elements of $a^? P u^* \meet' \varepsilon = \{\varepsilon\}$.  Since
  we remove subsumed word products, as per Remark~\ref{rem:meet'}, the
  latter elements do not occur.  The elements that remain are of the
  form ${c}^? R$ or just $R$, where $R \in P u^* \meet' {v}^*$ has the
  required form by induction hypothesis.
  
  In case (2), we use (\ref{eq:inter:**}): the elements of
  $u_0^* P u^* \meet' {v}^*$ are of the form ${(u_0\meet v)}^* R$
  where
  $R \in (P u^* \meet' {v}^*) \cup (u_0^* P u^* \meet' \varepsilon) =
  (P u^* \meet' {v}^*) \cup \{\varepsilon\}$.  Let us enumerate
  $P u^* \meet' {v}^*$: by induction hypothesis, we can write its
  elements as $R_1 {(u\meet v)}^*$, \ldots, $R_n {(u\meet v)}^*$.  We
  note that $n \geq 1$, because $\SEval {Pu^*}' \cap \SEval {{v}^*}$
  is non-empty: indeed, that intersection contains the empty word
  $\epsilon$.  Then the element ${(u_0 \meet v)}^* \varepsilon$ of the
  set $u_0^* P u^* \meet' {v}^*$ is subsumed by, say,
  ${(u_0\meet v)}^* R_1 {(u\meet v)}^*$, and will be removed,
  following Remark~\ref{rem:meet'}.  The only elements that remain are
  ${(u_0\meet v)}^* R_i {(u\meet v)}^*$, when $i$ varies over some
  subset of $\{1, \cdots, n\}$, and they are of the required form.
\end{proof}

Instead of redesigning an S-representation for $X^{\leq\omega}$ (or
$X^\omega$) from scratch, this allows us to reuse most of what we know
for $X^*$.  Item~(3) below is justified by
Lemma~\ref{lemma:irred:subset}, and item~(5) is by
Lemma~\ref{lemma:irred:inter} (resp.,
Lemma~\ref{lemma:irred:inter:infinite}), refined using
Lemma~\ref{lemma:PF*:inter:Srepr} (i.e., every element of
$P u^* \meet' Q {v}^*$ is of the form $R {(u\meet v)}^*$ for some $R$,
where $u\meet v \eqdef \bigcup_{a \in u, b \in v} (a \meet b)$).
\begin{thm}
  \label{thm:srepr:Xomega}
  Given an S-representation
  $(S, \SEval \_, \unlhd, \allowbreak \topcap, \meet)$ of a Noetherian
  space $X$, let $(S', \SEval \_', \unlhd', \topcap', \meet')$ be the
  S-representation of $X^*$ given in Proposition~\ref{prop:srepr:X*},
  with the optimization of Remark~\ref{rem:meet'}.  Then the following
  tuple $(S'', \SEval \_'', \unlhd'', \topcap'', \meet'')$ is an
  S-representation of $X^{\leq\omega}$ (resp., $X^\omega$):
  \begin{enumerate}
  \item $S''$ is the collection of pairs $(P, u)$ where $P \in S'$ and
    $u$ is a finite (resp., and non-empty) subset of $S$.
  \item
    $\SEval {(P, u)}'' = \SEval {P}' (\bigcup_{a \in u} \Eval
    a)^{\leq\omega}$ (resp.,
    $\SEval {P}' (\bigcup_{a \in u} \Eval a)^\omega$).
  \item $(P, u) \unlhd'' (Q, v)$ if and only if $P u^* \unlhd' Q
    {v}^*$ and for every $a \in u$, there is an $b \in v$ such that
    $a \unlhd b$.
  \item $\topcap''$ is $(\varepsilon, \topcap)$.
  \item $(P, u) \meet'' (Q, v)$ is defined as follows: writing
    $u\meet v$ for $\bigcup_{a \in u, b \in v} (a \meet b)$,
    $P u^* \meet' Q {v}^*$ is a finite set of finite word products of
    the form $R_i {(u\meet v)}^*$, $1\leq i\leq n$, and
    $(P, u) \meet'' (Q, v)$ is then the set of pairs
    $(R_i, u\meet v)$, $1\leq i\leq n$ (resp., the same formula if
    $u\meet v \neq \emptyset$, otherwise $\emptyset$).  \qed
  \end{enumerate}
\end{thm}

\section{Final notes}
\label{sec:final-notes}

\paragraph{\emph{Related work.}}

We must cite Simon Halfon's PhD thesis \cite{Halfon:PhD}, and
especially Section~9.1 there.  Our study of $X^\omega$ is very close
to his.  At first glance, it may seem that we add some generality to
his study, in the sense that Halfon studies $X^\omega$ (as a
preordered set) in the special case where $X$ is an $\omega^2$-wqo.
In that case, $X^\omega$ is wqo (as a set preordered by
$\leq^\omega$).

In a world of preorders, it is natural to replace sobrifications by
ideal completions.  Indeed, the ideal completion of a preordered set
$X$ coincides with the sobrification of $X$, provided that $X$ is
given its Alexandroff topology.  Halfon obtains that the ideal
completion of $X^\omega$ (as a preordered set) is characterized in
terms of $\omega$-regular expressions that are similar to the infinite
word products we introduce in Section~\ref{sec:sobrification-xomega},
although slightly more complicated, as the $F^\omega$ part of
$\omega$-regular expressions no longer involves elements $F$ of
$\Hoare (X)$ but ideals of $\Hoare (X)$.  The mismatch is due to the
fact that our space $X^\omega$ will almost never have an Alexandroff
topology (unless $X$ is essentially finite, see
Proposition~\ref{prop:words:wqo}), and therefore the ideal completion
of $X^\omega$ in general differs from $\Sober (X^\omega)$ (where
$X^\omega$ is given the asymptotic subword topology, as we do, not the
Alexandroff topology of $\leq^\omega$), for every $\omega^2$-wqo $X$
that is not essentially finite.

\paragraph{\emph{Other initial maps.}}

Our study of $X^{\leq\omega}$ (resp., $X^\omega$) proceeds by finding
an initial map $\langle \pref, \suf \rangle$ from $X^{\leq\omega}$ to
the more familiar space $\Sober (X^*) \times \Hoare (X)$.  This has
notable advantages.  For example, the fact that $X^{\leq\omega}$ (and
its subspace $X^\omega$) is Noetherian if and only if $X$ is follows
immediately from previously known results on sobrifications, on the
Hoare powerspace, and on spaces of \emph{finite} words.  We took this
further in the study of S-representations of $X^{\leq\omega}$ (resp.,
$X^\omega$), where we insisted on \emph{reducing} the question to
S-representations for finite words (and powersets).  We could have
computed intersections of infinite word products directly, notably,
but we feel that would have been less interesting.

Remarkably, there are many other initial maps that we could have used
instead of $\langle \pref, \suf \rangle$.  The advantage of the latter
is that it shows how the asymptotic subword topology splits into the study of
finite chunks of information (prefixes) and infinite behaviors
(suffixes).  Here are two different initial maps that we could have
used.

The first one is the composition:
\[
  \xymatrix{
    X^{\leq\omega} \ar[r]^(0.3){\langle \pref, \suf \rangle}
    & \Sober (X^*) \times \Hoare (X)
    \ar[r]^(0.35){\identity \relax \times j}
    & \Sober (X^*) \times \Sober (X^*)
    \cong
    \Sober (X^* \times X^*)
    \ar[r]^(0.65){\Sober (c)}
    & \Sober ((X+X)^*)
  }
\]
where $j \colon \Hoare (X) \to \Sober (X^*)$ maps $F$ to $F^*$, and
$c$ maps every pair of finite words $(a_1 \cdots a_m, b_1 \cdots b_n)$
to the word
$\iota_1 (a_1) \cdots \iota_1 (a_m) \iota_2 (b_1) \cdots \iota_2
(b_n)$, where $\iota_1$, $\iota_2$ are the two canonical injections of
$X$ into $X+X$.  Note that
$j^{-1} (\diamond \langle U_1 ; U_2 ; \cdots U_n \rangle) = \Diamond
U_1 \cap \Diamond U_2 \cap \cdots \cap \Diamond U_n$, showing that $j$
is initial.  As for $c$, every open subset of $X+X$ can be written as
$U+V$ where $U$ and $V$ are open in $X$, and we have
$c^{-1} (\langle U_1+V_1; \cdots ; U_n+V_n \rangle) = \bigcup_{k=0}^n
\langle U_1 ; \cdots ; U_k \rangle \times \langle V_{k+1} ; \cdots
; V_n \rangle$, showing that $c$ is continuous, and as special case
(with $n=2$, $V_1$ and $U_2$ empty) that
$\langle U_1 \rangle \times \langle V_2 \rangle = c^{-1} (\langle
U_1+\emptyset; \emptyset+V_2 \rangle)$, which allows us to conclude
that $c$ is initial with the help of Remark~\ref{rem:initial}.  This
implies that $\Sober (c)$ is initial, hence that the whole composition
shown above is initial, too.  The point of using this, as an
alternative to $\langle \pref, \suf \rangle$, is to realize that using
the Hoare powerspace is not required at all, and that the study of
$X^\omega$ reduces to the study of the sobrification of a space of
finite words only, on the extended alphabet $X+X$.

A second alternative to $\langle \pref, \suf \rangle$ is the following
map $q \colon X^{\leq\omega} \to \Sober ((X+X)^*)$ (see
Appendix~\ref{sec:q-initial} for a proof that $q$ is initial, when $X$
is Noetherian).  For short, let us write $-a \eqdef (0, a)$ for every
$a \in X$, $+a \eqdef (1, a)$, $\pm A \eqdef \{-a, +a \mid a \in A\}$
for every $A \subseteq X$.  For every $w \in X^{\leq\omega}$ and every
$n \in \dom w$, let
$q_n (w) \eqdef (\dc -w_0)^? (\dc -w_1)^?  \cdots (\dc -w_n)^?
\allowbreak (cl (\pm \{w_m \mid m \geq n+1\}))^*$.  The sequence
${(q_n (w))}_{n \in \nat}$ is a descending sequence of (irreducible)
closed sets.  When $X$ is Noetherian, there must therefore be an index
$n_0 \in \nat$ such that $q_n (w) = q_{n_0} (w)$ for every
$n \geq n_0$, and we define $q (w)$ as $q_{n_0} (w)$.  Notice the
similarity with Remark~\ref{rem:pref:PF*}.  Note also that $q$ is
slightly different from our previous alternative, which maps $w$ to
$(\dc -w_0)^? (\dc -w_1)^?  \cdots (\dc -w_n)^? (cl (\{-w_m \mid m
\geq n+1\}))^* (cl (\{+w_m \mid m \geq n+1\}))^*$ instead.  A similar
approach will turn out to be the right one in our study of infinite
trees (which should be part~III of this work).

\paragraph{\emph{Transfinite sequences.}}

We have dealt with the space $X^\omega$, but what would be a proper,
analogous treatment of spaces of sequences of length $\alpha$, for an
arbitrary (or countable) indecomposable ordinal $\alpha$?  The bqo
theory of such preordered sets is well-known \cite{NW:bqo}.  We will
deal with that aspect in part~II.

\subsection*{Acknowledgements}

I would like to thank the anonymous referee for her/his valuable
comments, in particular for insisting on some small issues that I
missed repeatedly.

\DeclareRobustCommand{\VAN}[3]{#3}

\bibliographystyle{abbrv}
\ifarxiv

\else
\bibliography{noeth}
\fi

\appendix

\section{\texorpdfstring{$q$}{q} is initial}
\label{sec:q-initial}

We use an alternate definition of $q$.  Given any $w \in X^\omega$,
let $A_n \eqdef \{w_m \mid m \geq n+1\}$.  Then
${(cl (\pm A_n))}_{n \in \nat}$ is a descending sequence of closed
subsets of $X+X$.  If $X$ is Noetherian, then there must be an index
$n_1$ such that for every $n \geq n_1$,
$cl (\pm A_n) = cl (\pm A_{n_1})$.  We pick $n_1$ larger than or equal
to the $n_0$ given in the definition of $q$.  Then
$q (w) \eqdef (\dc -w_0)^? (\dc -w_1)^?  \cdots (\dc -w_n)^? (cl
(\pm A_n))^*$ for every $n \geq n_1$, by definition of $q$.

We proceed and show that $q$ is continuous.

For that, we claim that: $(*)$
$q^{-1} (\diamond \langle U_1\allowbreak + V_1 ; \cdots ;
U_\ell+V_\ell \rangle)$ is equal to
$\bigcup_{k=0}^\ell \langle U_1 ; \cdots ; U_k ; \infy {(U_{k+1} \cup
  \allowbreak V_{k+1})} \cap \cdots \cap \infy {(U_\ell \cup V_\ell)}
\rangle$, where $U_1$, $V_1$, \ldots, $U_\ell$, $V_\ell$ are arbitrary
open subsets of $X$.

Let $w \in X^{\leq\omega}$, let us fix $n \eqdef n_1$ in the
definition of $q (w)$, and let us imagine that $q (w)$ is in
$\diamond \langle U_1\allowbreak + V_1 ; \cdots ; U_\ell+V_\ell
\rangle$.  There are letters $a_1 \in U_1+V_1$, \ldots,
$a_\ell \in U_\ell + V_\ell$, and indices $k$ ($0 \leq k \leq \ell$)
and $i_1 < \cdots < i_k$ between $0$ and $n_1-1$ such that
$a_1 \leq -w_{i_1}$, \ldots, $a_k \leq -w_{i_k}$, and $a_{k+1}$,
\ldots, $a_\ell$ are all in $cl (\pm A_{n_1})$.  In particular, every
$U_i+V_i$ with $i \geq k+1$ intersects $cl (\pm A_{n_1})$---which is
equal to $cl (\pm A_n)$ for every $n \geq n_1$---hence also $\pm A_n$
for every $n \geq n_1$.  This means that there are infinitely many
indices $n \geq n_1$ such that $-w_n$ or $+w_n$ is in $U_i + V_i$, in
particular such that $w_n \in U_i \cup V_i$, and that holds for every
$i \geq k+1$.  Therefore $w$ is in
$\langle U_1 ; \cdots ; U_k ; \infy {(U_{k+1} \cup \allowbreak
  V_{k+1})} \cap \cdots \cap \infy {(U_\ell \cup V_\ell)} \rangle$.

In the reverse direction, let $w$ be in
$\langle U_1 ; \cdots ; U_k ; \infy {(U_{k+1} \cup \allowbreak
  V_{k+1})} \cap \cdots \cap \infy {(U_\ell \cup V_\ell)} \rangle$ for
some $k$, $0\leq k \leq \ell$.  Let us write $w$ as $us$ where
$u \in \langle U_1 ; \cdots ; U_k \rangle$ and (if $\ell > k$) $s$
contains infinitely many letters from each $U_i \cup V_i$,
$k+1 \leq i \leq \ell$.  There are letters $a_1 \in U_1$, \ldots,
$a_k \in U_k$ such that $a_1 \cdots a_k$ is a subword of $u$.  We
write again $q (w)$ as
$(\dc -w_0)^? (\dc -w_1)^?  \cdots (\dc -w_{n-1})^? (cl (\pm A_n))^*$,
with $n \geq n_1$ arbitrary.  We pick such an $n$ so that it exceeds
the length of $u$.  This way, the finite word $(-a_1) \cdots (-a_k)$
is in $(\dc -w_0)^? (\dc -w_1)^?  \cdots (\dc -w_{n-1})^?$.  For each
$i$, $k+1 \leq i \leq \ell$, since $s$ contains infinitely many
letters from $U_i \cup V_i$, so does $w$, and we can therefore find at
least one of the form $w_m$ with $m \geq n$, hence in $A_n$.  This
implies that $A_n$ intersects $U_i \cup V_i$.  We pick a letter $b_i$
in the intersection, for each $i$ with $k+1 \leq i \leq \ell$.  If
$b_i$ is in $U_i$, we let $c_i \eqdef -b_i$, otherwise
$c_i \eqdef +b_i$, so that $c_i$ is in $U_i+V_i$.  Then the word
$(-a_1) \cdots (-a_k) c_{k+1} \cdots c_\ell$ is in $q (w)$, and in
$\langle U_1 ; \cdots ; U_k ; U_{k+1} \cup \allowbreak V_{k+1}; \cdots
; U_\ell \cup V_\ell \rangle$.  It follows that $q (w)$ is in
$\diamond \langle U_1\allowbreak + V_1 ; \cdots ; U_\ell+V_\ell
\rangle$.

That finishes to prove $(*)$, hence that $q$ is continuous.

Specializing $(*)$ to the case where $V_1$, \ldots, $V_j$, $U_{j+1}$,
\ldots, $U_\ell$ are empty (for some arbitrary $j$,
$0 \leq j \leq \ell$), the terms
$\langle U_1 ; \cdots ; U_k ; \infy {(U_{k+1} \cup \allowbreak
  V_{k+1})} \cap \cdots \cap \infy {(U_\ell \cup V_\ell)} \rangle$
with $k \geq j+1$ are all empty (because $U_k$ is empty).  The same
terms with $k \leq j$ instead are of the form
$\langle U_1 ; \cdots ; U_k ; \infy {U_{k+1}} \cap \cdots \cap \infy
{U_j} \cap \infy {V_{j+1}} \cap \cdots \cap \infy {V_\ell} \rangle$,
and it is easy to see that they are all included in the term obtained
when $k = j$, namely
$\langle U_1 ; \cdots ; U_j ; \infy {V_{j+1}}
\cap \cdots \cap \infy {V_\ell} \rangle$.   It follows that
$q^{-1} (\diamond \langle U_1\allowbreak + \emptyset ; \cdots ; U_j +
\emptyset; \emptyset+V_{j+1} ; \cdots \emptyset+V_\ell \rangle)$ is
equal to
$\langle U_1 ; \cdots ; U_j ; \infy {V_{j+1}} \cap \cdots \cap \infy
{V_\ell} \rangle$.  The latter is the general form of the basic open
subsets on $X^{\leq\omega}$ given in Lemma~\ref{lemma:limsup:base}.
Using Remark~\ref{rem:initial}, $q$ is initial.















\end{document}